\documentclass[12pt]{article}

\usepackage{amsmath,amsfonts,amssymb,amsthm,graphicx,float,color,verbatim}
\usepackage{caption}
\usepackage{subfig}
\newcommand*\midpoint[1]{\overline{#1}}

\usepackage{geometry}
\theoremstyle{plain}
\newtheorem{thm}{Theorem}[section]
\newtheorem{lem}[thm]{Lemma}

\newtheorem{Theorem 1}{Theorem}

\theoremstyle{definition}
\newtheorem{defn}[thm]{Definition}

\theoremstyle{remark}
\newtheorem{rem}[thm]{Remark}

\title{No surface-knot of genus one has triple point number two}
\author{A. Al Kharusi\footnote{Department of Mathematics and Statistics, College of Science, Sultan Qaboos University, Email address: p53723@squ.edu.om}   \quad and \quad  T. Yashiro\footnote{Department of Mathematics and Statistics, College of Science, Sultan Qaboos University, Email address: yashiro@squ.edu.om}}
\date{}
\begin{document}
\maketitle
\begin{abstract}
It is known that there is no 2-knot with triple point number two. The present paper shows  that there is no surface-knot of genus one with triple point number two. In order to prove the result, we use Roseman moves and the algebraic intersection number of simple closed curves in the double decker set. 
\end{abstract}
\section{Introduction}
A \textit{surface-knot} $F$ is a (might be disconnected or non-orientable) closed surface smoothly embedded in the Euclidean $4$-space ${\mathbb R}^4$. It is called a \textit{2-knot} if it is homeomorphic to a $2$-sphere.  
The \textit{triple point number} of $F$ is analogous to the crossing number of a classical knot. Specifically, it is defined by the minimal number of triple points over all projections in ${\mathbb R}^3$ representing it, and it is denoted by $t(F)$.  Surface-knot tabulations based on the triple point numbers are considered in \cite{Hatakenaka,Oshiro,paper_2Sat1,paper_2Sat3,paper_2SatShim,paper_2SatShim2}. Up to now, we have very few examples of surface-knots whose triple point numbers are determined \cite{paper_2SatShim,paper_2SatShim2}.
A non-trivial surface-knot $F$ with $t(F)=0$ is called a \textit{pseudo-ribbon} \cite{paper_2Kaw} (if $F$ is a 2-knot, then it is called a \textit{ribbon 2-knot}). It is proved in \cite{paper_2Sat} that any surface-knot $F$ satisfies $t(F) \neq 1$. There are two known examples of a disconnected surface-knot $F$ consisting of two components with $t(F)=2$ \cite{Oshiro,paper_2Sat1}, none of these examples is orientable. S. Satoh showed in  \cite{paper_2Sat3} that no 2-knot has triple point number two or three. It has been proved in \cite{paper_2Yash} that if a connected orientable surface-knot has at most two triple points and the lower decker set is connected, then the fundamental group of the surface-knot is isomorphic to the infinite cyclic group. Till now, we have no examples of surface-knots with odd triple point number, even if the surface-knot is non-orientable, or disconnected.  The 2-twist-spun trefoil is known to have the triple point number four \cite{paper_2SatShim}. In particular, it is counted as one of the simplest non-ribbon 2-knots according to the triple point number.  This implies that if there exists an orientable surface-knot with triple point number two, then it must be with non-zero genus.  We show in this paper that it must be with genus at least two indeed. In particular, we show the following theorem.
\begin{thm}\label{xy}
Let $F$ be an orientable surface-knot of genus one. If the singularity set of the projection into 3-space $\mathbb{R}^3$ contains two triple points, then $F$ satisfies $t(F)=0$.
\end{thm}  
From Theorem \ref{xy}, we see that if $F$ is non-trivial, then there exists a projection of $F$ with singularity set consists of only simple closed double curves.\\
In this paper, a surface-knot is always assumed to be oriented. The paper is organized as follows. In section 2, we review some basics about the surface-knot diagrams. Roseman moves are recalled in section 3, in which we give descriptions of the moves $R$-$2$, $R$-$5$ and $R$-$7$. In section 4, we refer to the obstruction on the projection of a surface-knot found by S. Satoh \cite{paper_2Sat2}.  Section 5 reviews the algebraic intersection number of two loops in the torus. Section 6 provides some lemmas that are needed for the discussion in section 7, where the proof of the main result (Theorem \ref{xy}) is given.

\section{Preliminaries}
\subsection{Surface-knot diagrams}
In order to describe a surface-knot $F$, we consider the projection of the surface-knot into $\mathbb{R}^3$ with some extra information. This is a generalization of the notion of knot diagrams in classical knot theory.\\
Let $p:\mathbb{R}^{4} \rightarrow \mathbb{R}^{3}$ be the orthogonal projection map defined by $p(x_1,x_2,x_3,x_4)=(x_1,x_2,x_3)$. The image of a surface-knot $F$ under the projection, $p(F)$ in $\mathbb{R}^{3}$, is denoted by $\vert \Delta \vert$.  We may move $F$ in $\mathbb{R}^4$ slightly so that $\vert \Delta \vert$ becomes a generic surface. The closure of the multiple point set 
\[
\{x \in p(F) \mid p(x_1)=p(x_2)=x \quad \text{for some} \quad x_1 \neq x_2  \quad \text{where}  \quad x_1,x_2 \in F \}
\]
is called the \textit{singularity set} of the projected image and it consists of double points, isolated triple points and isolated branch points. Double points form a disjoint union of open arcs and simple closed curves. We say that such an open arc is called a \textit{double edge}. Both triple points and branch points are in the boundary of double edges. We will use the notations $\mathcal{D}$, $\mathcal{T}$, $\mathcal{B}$,  $\mathcal{E}$ to denote a double point, a triple point, a branch point and a double edge, respectively. We also denote by $\mathcal{M}_2$, $\mathcal{M}_3$ and $\mathcal{S} \subset p(F)$ the sets of all double points,  all triple points and all branch points, respectively. Let $h:\mathbb{R}^4 \rightarrow \mathbb{R}$ be the height function defined by $h(x_1,x_2,x_3,x_4)=x_4$. For a double point $\mathcal{D}$ in $\vert \Delta \vert$, there is a 3-ball neighbourhood $B^3(\mathcal{D})$ containing $\mathcal{D}$ such that $(p|_{F})^{-1}\big(B(\mathcal{D}) \cap \vert \Delta \vert \big)$ is a disjoint union of disks $D_U$ and $D_L$ in $F$ with $h(x)>h(x')$ holds for any $x \in D_U$ and $x' \in D_L$. We say that $p(D_U)$ and $p(D_L)$ are \textit{upper} and \textit{lower} sheets at $\mathcal{D}$, respectively, and denoted by $U$ and $L$, respectively. Similarly, for a triple point $\mathcal{T}$ in $\vert \Delta \vert$, there exists a 3-ball neighbourhood $B^3(\mathcal{T})$ of $\mathcal{T}$ in $\mathbb{R}^3$ such that $(p|_{F})^{-1}\big(B(\mathcal{T}) \cap \vert \Delta \vert\big)$ consists of three disjoint disks $D_T$, $D_M$ and $D_B$ in $F$ with $h(x)>h(x')>h(x'')$ holds
for any $x \in D_T$, $x' \in D_M$, and $x'' \in D_B$.  $p(D_T)$, $p(D_M)$, and $p(D_B)$ are labelled $T,M$ and $B$ and called the \textit{top} sheet, the \textit{middle} sheet and the \textit{bottom} sheet, respectively. 
A surface-knot diagrams is a generalization of the classical knot diagrams. That is, a surface-knot diagram of $F$, denoted by $\Delta$, is obtained from $\vert \Delta \vert$ in $\mathbb{R}^{3}$ by removing a small neighbourhood of the singularity set in lower sheets. In particular, in a diagram,  locally the lower sheet is divided into two regions and the middle and bottom sheets are broken into two and four regions, respectively. Thus a surface-knot diagram is represented by a disjoint union of compact surfaces which are called \textit{broken sheets} (cf.\cite{3}).  The three pictures in Figure \ref{singu} show broken sheets around a double point, a triple and a branch point from left to right, respectively.
\begin{figure}[H]
\centering
\captionsetup{font=scriptsize}      
\mbox{\includegraphics[scale=0.5]{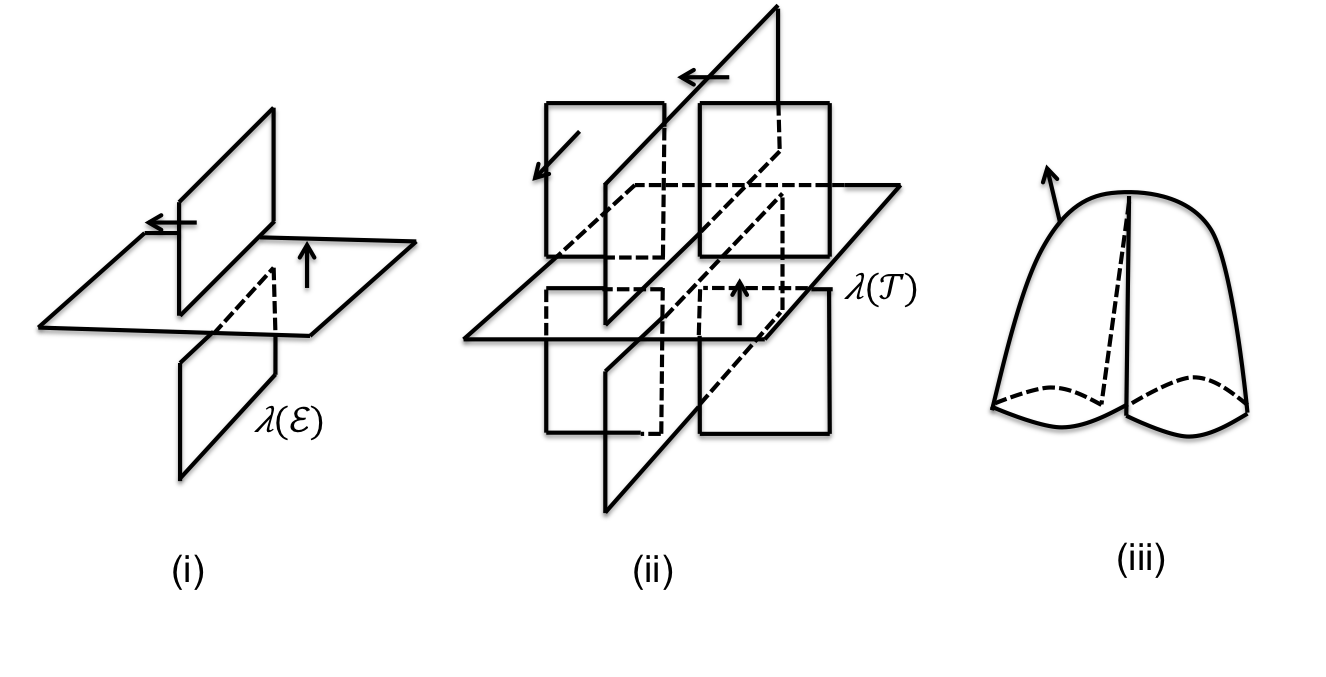}}
 \caption{}
\label{singu}
\end{figure}
\subsection{$t$-minimal diagrams}
Let $\Delta$ be a surface-knot diagram of a surface-knot $F$. Let $t(\Delta)$ denote the number of triple points of $\Delta$. We say that a surface-knot diagram $\Delta$ is \textit{t-minimal} if it is a surface-knot diagram with minimal number of triple points for all possible diagrams of $F$, that is $t(\Delta)=t(F)$.
\subsection{Alexander numbering} 
An \textit{Alexander numbering} for a surface-knot is a function that assigns an integer to each 3-dimensional complemantary region of the diagram as follows.  Two regions that are separated by a sheet are numbered consecutively;  the region into which a normal vector to the sheet points has the larger number  (for example, see    \cite{paper_2Kam2}). Such a number is  called the \textit{index} of the region. For each point $x \in \mathcal{M}_2, \mathcal{M}_3$, or $\mathcal{S}$, the integer $\lambda(x)$ is called the \textit{Alexander numbering} of $x$ and defined as the minimal Alexander index among the four, eight, or three regions surrounding $x$, respectively.  Equivalently, $\lambda(x)$ is the Alexander index of a specific region $R$, called a \textit{source region},  where all orientation normals to the bounded sheets point away from $R$ (see Figure \ref{Alex}). For a double edge $\mathcal{E}$, we use the notation $\lambda(\mathcal{E})=\lambda(\mathcal{D})$, $\mathcal{D} \in \mathcal{E}$ as the Alexander numbering $\lambda(\mathcal{D})$ is independent of the choice of the double point $\mathcal{D}$.   

\begin{figure}[H]
\centering
\captionsetup{font=scriptsize}      
\mbox{\includegraphics[scale=0.5]{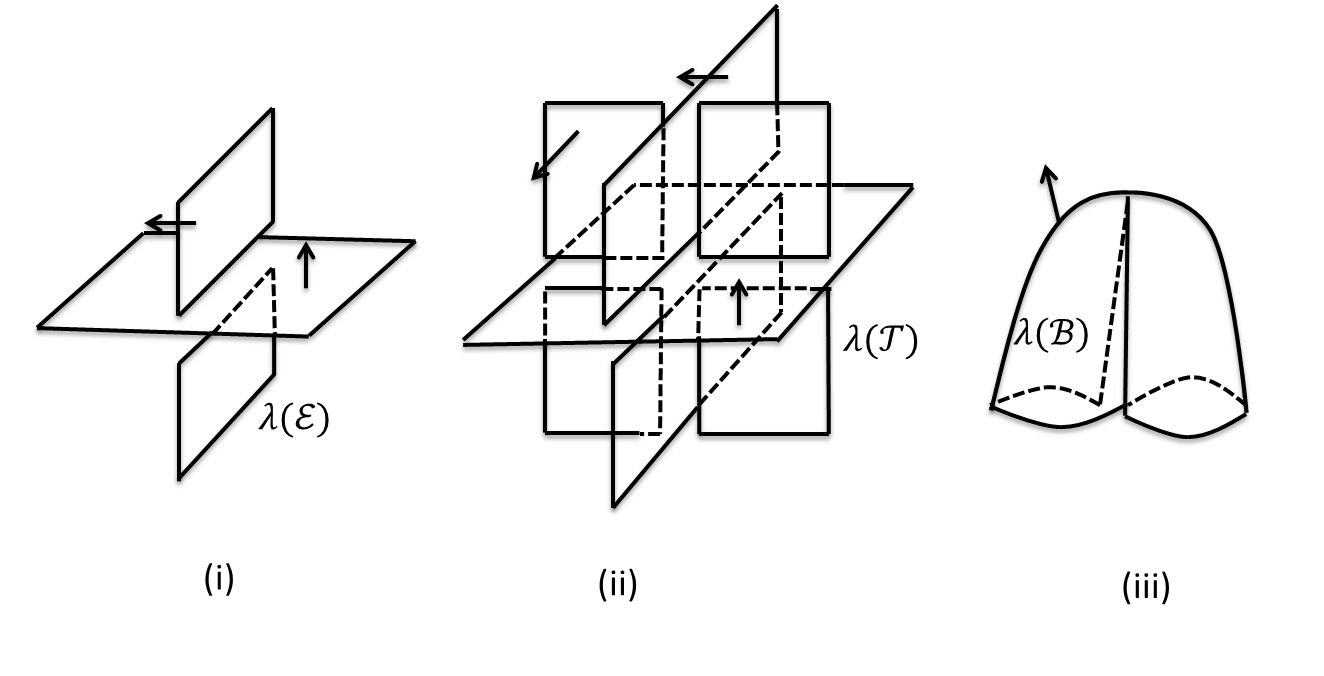}}
 \caption{}
\label{Alex}
\end{figure}
\subsection{Signs, orientations and type of branches at triple points}
We give sign to the triple point  $\mathcal{T}$ of a surface-knot diagram as follows. Let $n_T$ ,$n_M$ and $n_B$ denote the normal vectors to the top, the middle and the bottom sheets presenting their orientations, respectively. The \textit{sign} of  $\mathcal{T}$, denoted by $\epsilon({\mathcal{T}})$, is $+1$ if the triplet $(n_T,n_M,n_B)$ matches the orientation of $\mathbb{R}^3$ and otherwise $-1$. See Figure \ref{Alex} (ii), where the case of a positive triple point is depicted.\\
There are six double edges incident to $\mathcal{T}$ called the \textit{branches} of double edges at $\mathcal{T}$.  Such a branch is called a $b/m$-, $b/t$- or $m/t$-branch if it is the intersection of bottom
and middle, bottom and top, or middle and top sheets at $\mathcal{T}$, respectively. \\
We assign an orientation to a double edge in a surface-knot diagram so that for a tangent vector $v$ to the double edge at a double point $\mathcal{D}$, the ordered triple $(n_U,n_L,v)$ matches the orientation of $\mathbb{R}^3$, where $n_U$ and $n_L$ are normal vectors to the upper sheet $U$ and the lower sheet $L$ at $\mathcal{D}$ presenting their orientations, respectively.\\
Let $\mathcal{E} \subset \mathcal{M}_2$ be a double edge. Suppose one of the boundary points of $\mathcal{E}$ is a branch point $\mathcal{B}$. Because $\mathcal{E}$ connects to a branch point, it follows that $\lambda(\mathcal{E})=\lambda(\mathcal{B})$. The sign of the branch point $\mathcal{B}$, denoted by $\epsilon(\mathcal{B})\in \{+1,-1\}$, is defined according to the orientation of $\mathcal{E}$. In fact,  $\epsilon(\mathcal{B})=+1$ if the orientation of $\mathcal{E}$ points towards $\mathcal{B}$ and otherwise $-1$ (cf. \cite{sign}). Figure \ref{Alex} (iii) depicts a positive branch point.
\subsection{Double point curves of surface-knot diagrams}
The singularity set of a projection is regarded as a union of oriented curves immersed in $\mathbb{R}^3$. We call such an oriented curve a \textit{double point curve}. In the following we define the two kinds of double point curves in a diagram. 
Let $n_1<\dotso < n_k$ be an ordered sequence. Let $\mathcal{E}_{n_1},\dotso,\mathcal{E}_{n_k},\mathcal{E}_{n_{k+1}}=\mathcal{E}_{n_1}$ be double edges and let $\mathcal{T}_{n_1},\dotso,\mathcal{T}_{n_k},\mathcal{T}_{n_{k+1}}=\mathcal{T}_{n_1}$ be triple points of the surface-knot diagram $ \Delta $ of $F$. For $i=1,\dotso , k$, assume that 
\begin{itemize}
\item[(i)] $\mathcal{E}_{n_i}$ and $\mathcal{E}_{n_{i+1}}$ are in opposition to each other at $\mathcal{T}_{n_{i+1}}$, and 
\item[(ii)] $\mathcal{T}_{n_i}$ and $\mathcal{T}_{n_{i+1}}$ bound $\mathcal{E}_{n_i}$. 
\end{itemize}
Then the closure of the union $\mathcal{E}_{n_1} \cup...\cup \mathcal{E}_{n_k}$ forms a circle component called a \textit{double point circle} of the diagram.  Note that we do not assume  $\mathcal{T}_{n_i} \neq \mathcal{T}_{n_j}$, for distinct $i,j \in \{1, \dotso , k\}$. By giving a BW orientation to the singularity set (for a BW orientation see \cite{paper_2BW}), it is easy to verify the following.
\begin{lem}[\cite{paper_2BW}]\label{even}
The number of triple points along each double point circle is even.
\end{lem}

\begin{proof}
Let $C=\midpoint{\mathcal{E}}_{n_{1}} \cup \dotso \cup \midpoint{\mathcal{E}}_{n_{k}}$ be a double point circle of a surface-knot diagram, where $\midpoint{\mathcal{E}}$ stands for the closure of $\mathcal{E}$. We give a BW orientation to the singularity set such that the orientation restricted to branches at every triple point is as depicted in Figure \ref{ori}. It follows that the double branches $\mathcal{E}_{n_i}$ and $\mathcal{E}_{n_{i+1}}$ admit opposite orientations on both sides of $\mathcal{T}_{n_{i+1}}$. Hence $n$ is even.
\end{proof} 

\begin{figure}[H]
\centering
\captionsetup{font=scriptsize}      
\mbox{\includegraphics[scale=0.5]{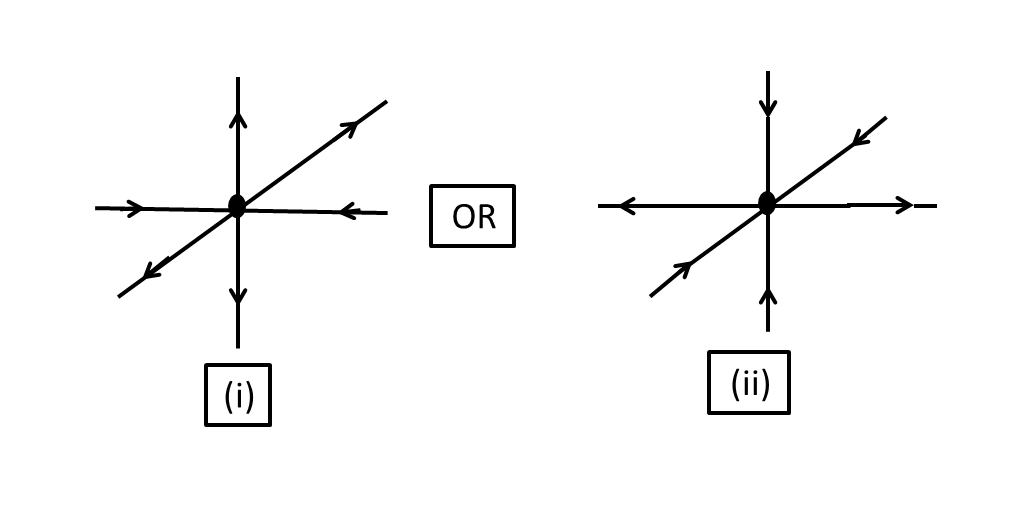}}
 \caption{BW orientation to the double branches at the triple point}
\label{ori}
\end{figure}

Similarly, we define a double point interval in a surface-knot diagram. Let $\mathcal{E}_{n_1},\dotso,\mathcal{E}_{n_k}$ be double edges, $\mathcal{T}_{n_1},\dotso,\mathcal{T}_{n_{k-1}}$ be triple points of $\Delta$ and suppose $\mathcal{B}_1$ and $\mathcal{B}_2$ are branch points of $\Delta$. Let the boundary points of $\mathcal{E}_{n_1}$ be the triple point $\mathcal{T}_{n_1}$ and the branch point $\mathcal{B}_1$. Let the double edge $\mathcal{E}_{n_k}$ be bounded by $\mathcal{T}_{n_{k-1}}$ and $\mathcal{B}_{2}$. Suppose that the following conditions hold:
\begin{itemize}
\item[(i)] $\mathcal{E}_{n_i}$ and $\mathcal{E}_{n_{i+1}}$ are in opposition to each other at $\mathcal{T}_{n_i}$ $(i=1,2,\dots,k-1)$, and
\item[(ii)] The double edge $\mathcal{E}_{n_i}$ is bounded by $\mathcal{T}_{n_{i-1}}$ and $\mathcal{T}_{n_i}$ $(i=2,\dots,k-1)$.
\end{itemize}
Then the closure of the union  $\mathcal{E}_{n_1} \cup \mathcal{E}_{n_2} \cup \dots \cup  \mathcal{E}_{n_k}$ forms an oriented interval component called a \textit{double point interval} of the diagram. Notice that the orientation of the double edges naturally leads to an orientation of a double point curve.
\begin{rem}\label{conn}
 Let $\Delta$ be a surface-knot diagram of a surface-knot $F$. Let $\mathcal{T}_{n_i}$, $\mathcal{T}_{n_{i+1}}$ and $\mathcal{T}_{n_{i+2}}$ be triple points giving order in a double point circle $C$ of $\Delta$. By giving an Alexander number to each of the eight regions surrounding $\mathcal{T}_{n_i}$, we see that it is impossible to have $\mathcal{T}_{n_{i}}=\mathcal{T}_{n_{i+1}}=\mathcal{T}_{n_{i+2}}$. 
\end{rem}

\subsection{Double decker sets of surface-knot diagrams}
The pre-image of the singularity set of a surface-knot diagram is called a \textit{double decker set} that is the union of \textit{upper} and \textit{lower decker sets} \cite{lift}. In particular, let $\Delta$ be a surface-knot diagram of a surface-knot $F$ and let $C$ be a double point curve of $\Delta$. For a double edge $\mathcal{E}$ contained in $C$. Let $(p \mid_ {F})^{-1}\big(\mathcal{E} \big)=\{\mathcal{E}^U,\mathcal{E}^L\}$ be a pair of open arcs such that $\mathcal{E}^U$ is in the upper disk $D_U$ while $\mathcal{E}^L$ is in the lower disk $D_L$ of $F$. Let $\bar{\mathcal{E}}$ stand for the closure of $\mathcal{E}$. Then the union $C^U=\bigcup_{\mathcal{E} \in C} \big(\bar{\mathcal{E}}^U
\big)$ is called the \textit{upper decker curve}. Similarly, the union $C^L=\bigcup_{\mathcal{E} \in C} \big(\bar{\mathcal{E}}^L
\big)$ is called the \textit{lower decker curve}. In fact, the upper or lower decker curve can be regarded as a circle or interval component immersed into $F$. The crossing point corresponds to a triple point in the projection. We use the notation $\mathcal{T}^W$ $W=\{T,M,B\}$ to indicate the pre-image of the triple point $\mathcal{T}$ in the $D_W$ disk. The union of upper decker curves forms the upper decker set. Similarly, the union of the lower decker curves gives the lower decker set.

\section{Roseman moves}
D. Roseman introduced seven types of local transformations and he proved the following lemma.
\begin{lem}[\cite{Roseman}] Two surface-knot diagrams are equivalent if and only if there exists a finite sequence of local moves to deform one diagram into the other.
\end{lem}
We call the local deformations \textit{Roseman moves} or \textit{moves}. Seven types of Roseman moves in \cite{Roseman} can be described by seven moves shown in Figure \ref{Fig(1)} \cite{kaw, Yashiro1}. The deformation from the left to the right is called an $R$-$i^+$ move and the reverse direction is called an $R$-$i^-$ except $R$-$7$.
\begin{figure}[H]
\centering
\captionsetup{font=scriptsize}      
\mbox{\includegraphics[scale=0.4]{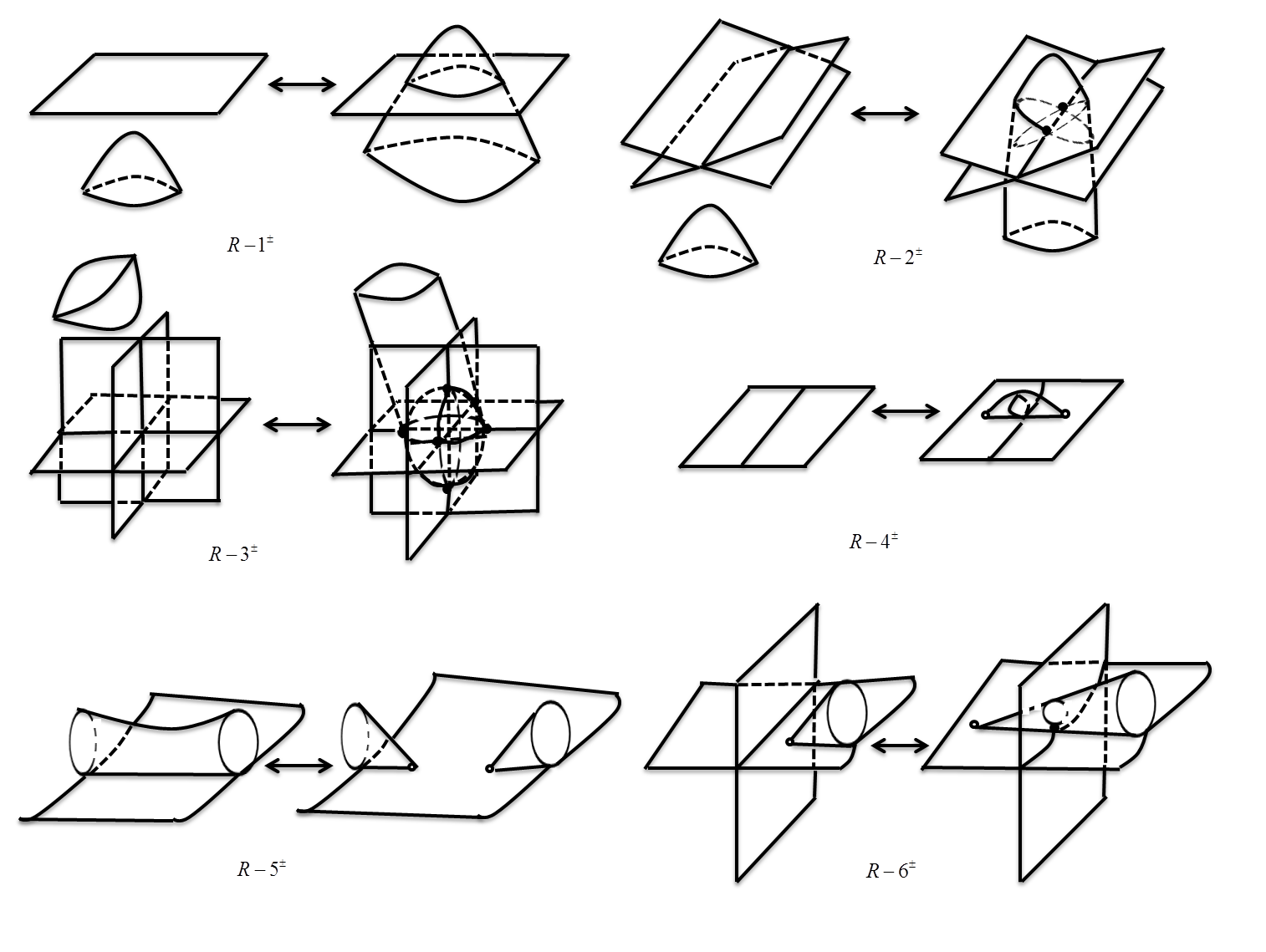}}
\label{4}
\end{figure}
\begin{figure}[H]
\centering
\captionsetup{font=scriptsize}      
\mbox{\includegraphics[scale=0.5]{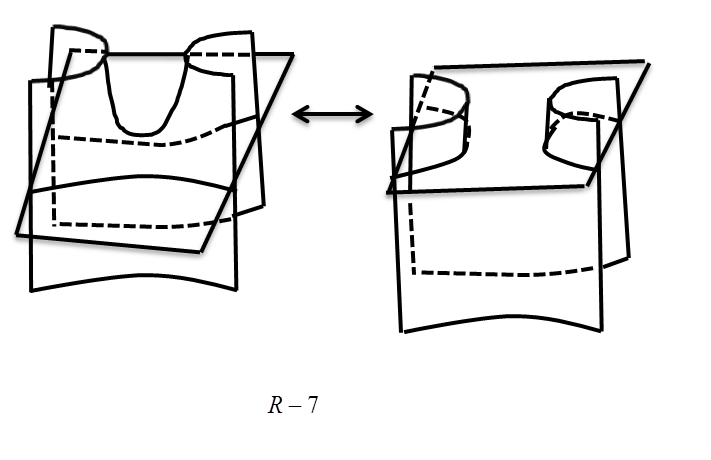}}
 \caption{Roseman moves}
\label{Fig(1)}
\end{figure}
\begin{lem}[\cite{paper_2SatShim}]\label{R6}
Let $\mathcal{E}$ be an edge of a surface-knot diagram whose boundary points are a triple point $\mathcal{T}$ and a branch point $\mathcal{B}$. If $\mathcal{E}$ is a $b/m$- or $m/t$-branch at $\mathcal{T}$, then the triple point $\mathcal{T}$ can be eliminated. 
\end{lem}
\begin{proof}
Since $\mathcal{E}$ is a $b/m$- or $m/t$-branch at $\mathcal{T}$,  we can apply the Roseman move $R$-$6^-$ to move the branch point along $\mathcal{E}$. As a result, $\mathcal{T}$ will be eliminated.   
\end{proof}
\subsection{$2$-cancelling pair of triple points}
We need to describe the $R$-$2^-$ move for proving some lemmas in section 6. The \textit{$2$-cancelling pair} is a pair of triple points that can be eliminated by applying the move $R$-$2^-$ indeed.  Let $(\mathcal{T}_1,\mathcal{T}_2)$ be a pair of triple points.  Let $\mathcal{E}_i$ $(i=1,2,3,4,5)$ be a double edge bounded by $\mathcal{T}_1$ and $\mathcal{T}_2$ such that $\mathcal{E}_i$ $(i=1,2,3,4,5)$ is of the same type at both $\mathcal{T}_1$ and $\mathcal{T}_2$. We arrange the double edges $\mathcal{E}_i$'s  so that $C_1=\midpoint{\mathcal{E}}_1 \cup \midpoint{\mathcal{E}}_2$ and $C_2=\midpoint{\mathcal{E}}_3 \cup \midpoint{\mathcal{E}}_4$ form two double point circles in $ \Delta $. Figure \ref{can} (b) shows the connection between the double edges $\mathcal{E}_i$'s  $(i=1,2,3,4,5)$.  In Figure \ref{can} (b), we ignore the over/under information. We consider all possible over/under information. 
\begin{figure}[H]
  \centering
  \subfloat[]{\includegraphics[width=0.2\textwidth]{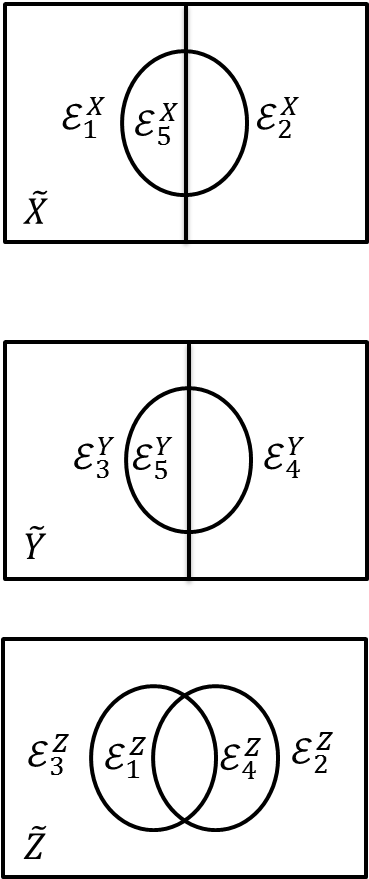}\label{fig:f2}}
   \hfill
   \subfloat[]{\includegraphics[width=0.6\textwidth]{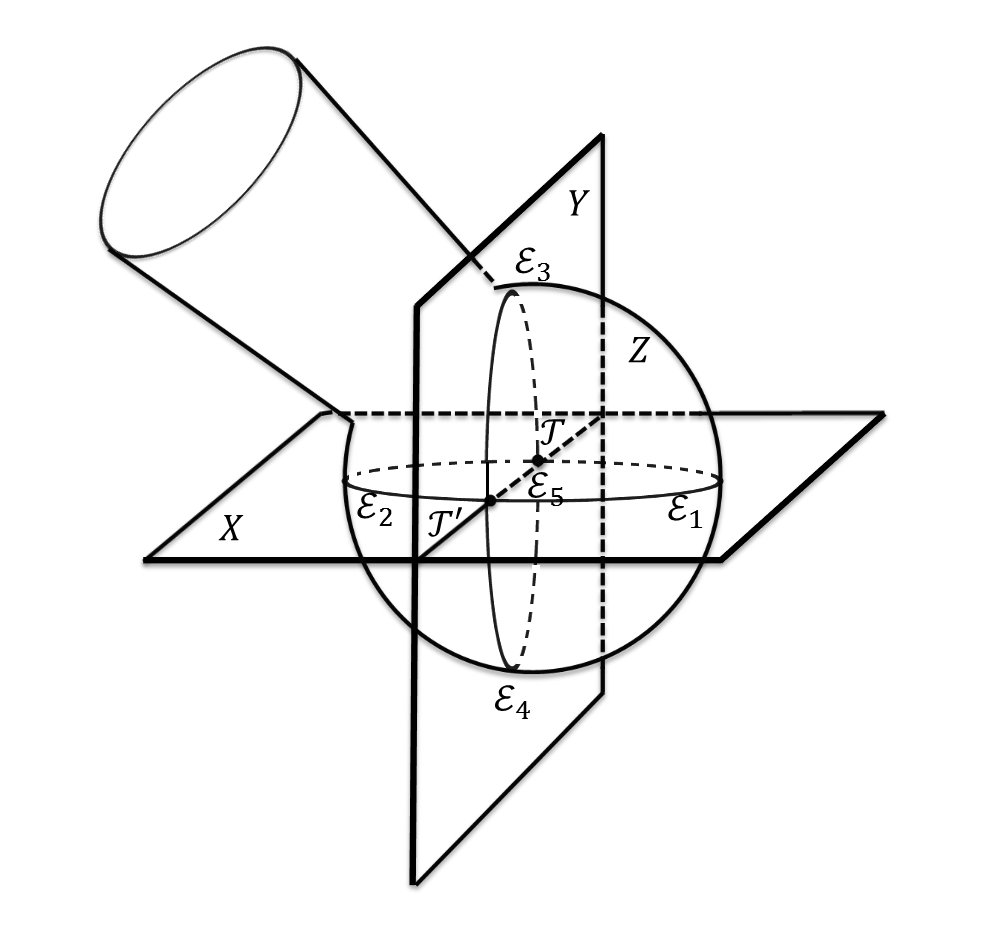}\label{fig:f1}}
  \caption{}
  \label{can}
\end{figure} 
Let $B^3(\mathcal{T}_1,\mathcal{T}_2)$ be a 3-ball in 3-space containing $\mathcal{E}_i$ $(i=1,2,3,4,5)$. Suppose that the pre-image $ (p\mid_{F})^{-1} \big(B^3(\mathcal{T}_1,\mathcal{T}_2) \cap \vert \Delta \vert \big)$ is a union of disjoint three sets $\widetilde{X},\widetilde{Y}$ and $\widetilde{Z} \subset F$.  We label the images of the sets $\widetilde{X},\widetilde{Y}$ and $\widetilde{Z}$ under the projection $p:\mathbb{R}^4 \rightarrow \mathbb{R}^3$ by $X$, $Y$ and $Z$ in $\Delta$, respectively as shown in Figure \ref{can}. By $\mathcal{E}^W$ $(W=X,Y,Z)$, we mean the pre-image of a double edge $\mathcal{E} \subset \Delta$ that is contained in $\widetilde{W}$  $(\widetilde{W}=\widetilde{X},\widetilde{Y},\widetilde{Z})$. Suppose that the following conditions hold:
 \begin{itemize}
 \item[(1)] In $\widetilde{X}$, the closure of $\mathcal{E}_i^{X} \cup \mathcal{E}_5^{X}$ $(i=1,2)$ bounds a disk such that the interior of the disk does not meet the double decker set.
 \item[(2)] In $\widetilde{Y}$, the closure of $\mathcal{E}_i^{Y} \cup \mathcal{E}_5^{Y}$ $(i=3,4)$  bounds a disk such that the interior of the disk does not meet  the double decker set.
 \item[(3)] In $\widetilde{Z}$, the closure of each of $\mathcal{E}_1^{Z} \cup \mathcal{E}_3^{Z}$, $\mathcal{E}_1^{Z} \cup \mathcal{E}_4^{Z}$ and $\mathcal{E}_2^{Z} \cup \mathcal{E}_4^{Z}$ is on the boundary of a disk such that the interior of the disk has empty intersection with the double decker set.
 \end{itemize}
A pair of triple points $(\mathcal{T}_1,\mathcal{T}_2)$ is called a \textit{$2$-cancelling pair} if and only if there is a 3-ball neighbourhood $B^3(\mathcal{T}_1,\mathcal{T}_2)$ in 3-space containing $\mathcal{T}_1$ and $\mathcal{T}_2$  such that the  pre-image $ (p\mid_{F})^{-1} \big(B^3(\mathcal{T}_1,\mathcal{T}_2) \cap \vert \Delta \vert \big)$ satisfies the conditions (1)-(3) above. 

\subsection{Descendent disks and pinch disks}
Let $J=[-1,1]$. Let $M_1$ be $J^2 \times \{0\} \subset J^3$. Let $M_2$ be $\{(x,y,z) \vert z=0.5x^2-2y^2+0.5\} \cap J^3$. The disk in the $yz$-plane bounded by the graphs $\{(0,y,0.5-2y^2) \vert y \in J\}$ and $\{(0,y,0) \vert y \in J\}$ will be denoted by $P_0$. \\
A disk $M$ embedded in $\mathbb{R}^3$ is a \textit{descendent disk} if there is a closed neighbourhood $N(M)$ of $M$ in $\mathbb{R}^3$ such that the pair $(N(M), N(M) \cap \vert \Delta \vert \cup M)$ is homeomorphic to  $(J^3,M_1 \cup M_2 \cup P_0)$ and satisfies the following properties:
\begin{itemize}                                                                                                                                                                                                                                                                                  
\item[(1)] $M \cap \vert \Delta \vert =\partial M=\{\lambda_1,\lambda_2\}$, where $\lambda_1$ and $\lambda_2$ are two simple arcs,  
\item[(2)] $\lambda_1 \cap \lambda_2=\{\mathcal{D}_1,\mathcal{D}_2\}$, where $\mathcal{D}_1$ and $\mathcal{D}_2$ are double points, and
\item[(3)] The double edges containing $\mathcal{D}_1$ and $\mathcal{D}_2$ have opposite orientation with respect to the orientation of the arc $\lambda_1$ or $\lambda_2$. 
\end{itemize}
The pair $(J^3,M_1 \cup M_2 \cup P_0)$ can be viewed as  a model of the neighbourhood of the descendent disk.
If a descendent disk exists, then the Roseman move $R_7$ can be applied. a Subset of $M_2$ is deformed along the descendent disk so that the connection of the double edges is changed and this change is correspondence to a band move.\\
We define embedded cylinder in $\mathbb{R}^3$ with radius $r_i$ by
\[
C(r_i)=\{(x,y,z) | (y-a)^2+(z-b)^2=r_i^2, (0,a,b) \in \text{int}(P_0) \} \cap J^3
\]
Let $\mathcal{C}=\cup _{i=1}^nC(r_i)$ be a finite union of pairwise disjoint closed cylinders embedded in $\mathbb{R}^3$.
\begin{lem}\label{50}
In the notation above, assume that $Q$ is a disk embedded in 3-space such that $Q$ has a closed 3-ball neighbourhood $N(Q)$ with the pair $(N(Q), N(Q) \cap \vert \Delta \vert \cup Q)$ is homeomorphic to $(J^3,M_1\cup M_2 \cup \mathcal{C} \cup P_0)$. Then $\Delta$ can be deformed into $\Delta'$ fixing outside $N(Q)$ such that there is a descendent disk $M \subset Q$ in the closed set of connected regions of $\mathbb{R}^3 \setminus \vert \Delta'\vert$. 
\end{lem}
\begin{proof}
If $\mathcal{C}$ contains a single cylinder $C$, then we can perform a deformation on $N(Q)$ as shown in Figure \ref{des} schematically, where $M$ is indicated by the shaded region. We see that this deformation is a combination of the Roseman moves $R$-$1^+$ and $R$-$7$. If $\mathcal{C}$ contains more than one cylinder,  then we see that $Q$ contains disjoint circles that are the intersection with the cylinders and $Q$. Let $d_0$ be the inner most circle  that is contained in the circle $d_1$ as shown in Figure \ref{des2}. We apply a procedure called Procedure I, to move $d_0$ out from the modified $Q$. Procedure I consists of the following three steps: (1) Take a simple arc $\gamma$ from a point $q_0$ on $\lambda_2$ to a point $q_1$ on $d_0$ such that the intersection of $\gamma$ and circles is the minimum; (2) Move a small disk neighbourhood of $q_0$ in $\Delta$ along $\gamma$ and apply $R$-$1^+$ move when it is needed so that the finger reaches $d_0$, (3) Apply the $R$-$7$ move at the inner most circle $d_0$ so that the modified $Q$ does not include $d_0$. If $d_1$ contains another inner most circle, then we repeat the same process to the modified $\lambda_2$ and the innermost circle as illustrated in Figure \ref{des3}. After all the inner most circles contained in $d_1$ are moved away from the modified $Q$, we apply the Roseman move $R$-$7$ in order to move the circle $d_1$. We repeat Procedure I and $R$-$7$ move as needed till the modified $Q$, denoted by $M$, becomes a descendent disk. 
\end{proof}
\begin{figure}[H]
\centering
\captionsetup{font=scriptsize}      
\mbox{\includegraphics[scale=0.32]{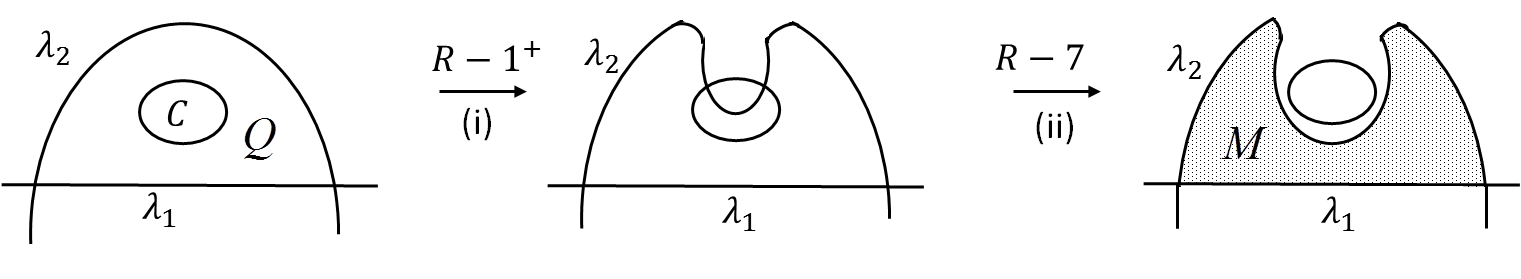}}
 \caption{}
 \label{des}
\end{figure}
\begin{figure}[H]
\centering
\captionsetup{font=scriptsize}      
\mbox{\includegraphics[scale=0.3]{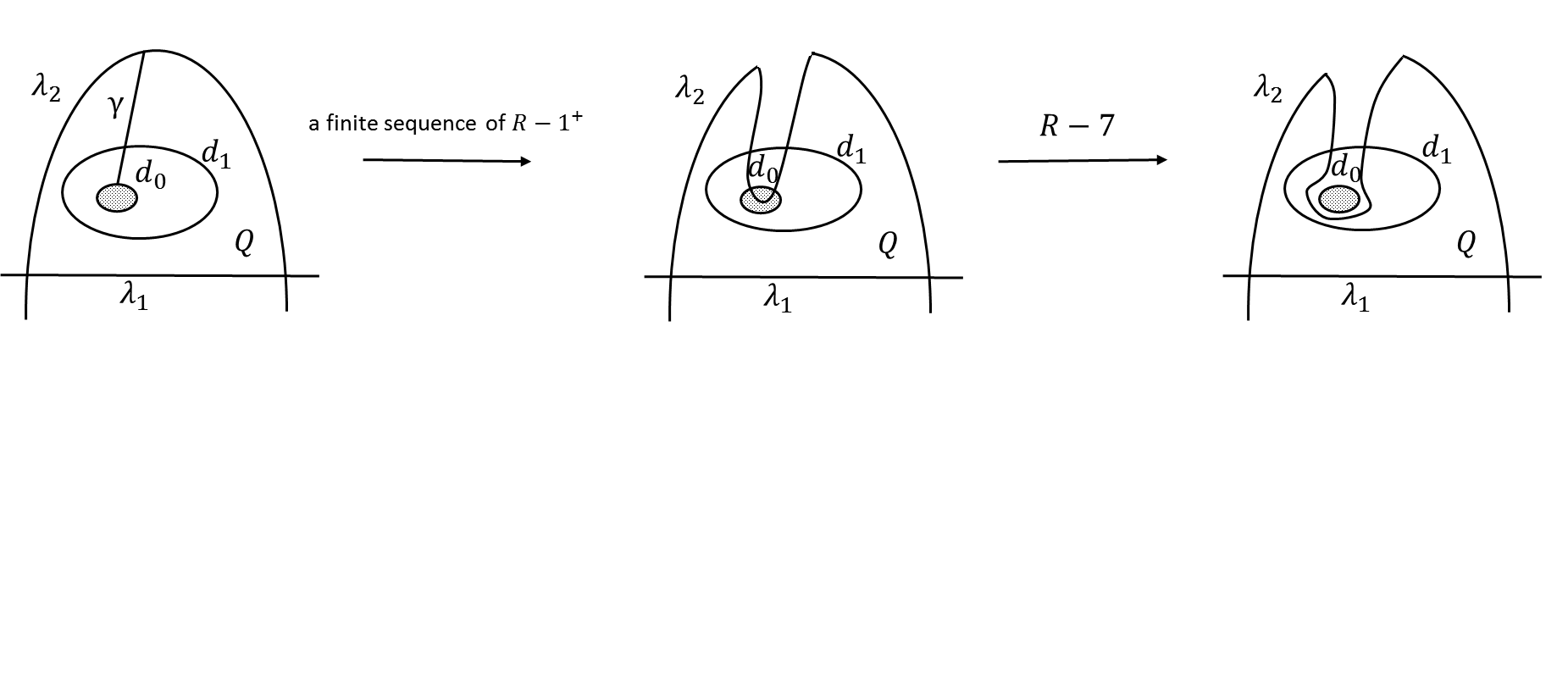}}
 \caption{Procedure I}
 \label{des2}
\end{figure}
\begin{figure}[H]
\centering
\captionsetup{font=scriptsize}      
\mbox{\includegraphics[scale=0.32]{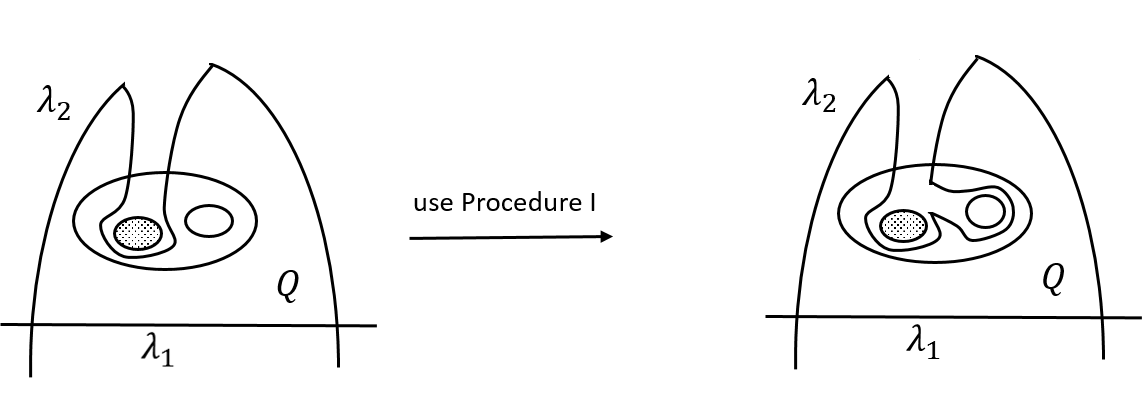}}
 \caption{}
 \label{des3}
\end{figure}
Note that the operation done above to $\Delta$ generates new double edges but never creates triple points. \\ 
Let $I=[0,1]$. Let $\lambda : I \rightarrow J^2$ be an immersion with only one crossing point such that $\lambda(0)=(-1,-1)$, $\lambda(1)=(1,-1)$ and $\lambda(1/4)=\lambda(3/4)=(0,0)$. The loop $\lambda(1/4) \times \{0\}=\lambda(3/4) \times \{0\}$ bounds a disk $P_0$ in $J^2 \times \{0\}$. \\
An embedded disk $P$ in $\mathbb{R}^3$ is a \textit{pinch disk} if there is a closed neighbourhood $N(P)$ of $P$ in $\mathbb{R}^3$ such that the pair $(N(P),N(P) \cap \vert \Delta \vert \cup P)$ is homeomorphic to  $(J^2 \times J, \lambda(I) \times J \cup P_0)$ and satisfies the following properties:
\begin{itemize}
\item[1.] $P \cap \vert \Delta \vert = \partial P$, and 
\item[2.]  $P$ is transversal to $\Delta$ along $\partial P$.
\end{itemize}
The pair $(J^2 \times J, \lambda(I) \times J \cup P_0)$ can be regarded as a model of the neighbourhood of  a pinch disk.
The existence of a pinch disk leads to a deformation of the surface-knot diagram such that a pair of branch points is created. This deformation is correspondence to the Roseman move $R_5^+$.\\
Let $\mathcal{C}=\cup _{i=1}^nC(r_i)$ denote a finite union of closed cylinders embedded in $\mathbb{R}^3$ such that for $i=1,\dots, n$, $C(r_i)$ is defined by $\{(x,y,z) | (x-a)^2+(y-b)^2=r_i^2, (a,b,0) \in \text{int}(P_0) \} \cap J^2 \times J$. 
\begin{lem}\label{51}
In the notation above, assume that $Q$ is a disk embedded in 3-space such that $Q$ has a closed 3-ball neighbourhood $N(Q)$ with the pair $(N(Q), N(Q) \cap \vert \Delta \vert \cup Q)$ is homeomorphic to $(J^2 \times J, \lambda(I) \times J \cup \mathcal{C} \cup P_0)$. Then $\Delta$ can be deformed into $\Delta'$ fixing outside $N(Q)$ such that there is a pinch disk $P \subset Q$ in the closed set of connected regions of $\mathbb{R}^3 \setminus \vert \Delta'\vert$. 
\end{lem}
\begin{proof}
The proof is similar to Lemma \ref{50}.
\end{proof}

\section{Numbers of triple points}
Suppose $\Delta$ is a $t$-minimal surface-knot diagram of the surface-knot $F$. Let $\mathcal{T}$ be a triple point of $\Delta$. From Lemma \ref{R6}, the other endpoint of any of the $b/m$- or $m/t$-branches at $\mathcal{T}$ must be a triple point. We classify triple points of $\Delta_t$ according to the other boundary points of the $b/t$-branches. At $\mathcal{T}$, the sheet transverse to the $b/t$-branches is the middle sheet.  Let $\mathcal{E}_1$ and $\mathcal{E}_2$ be the $b/t$-branches at  $\mathcal{T}$ such that the orientation normal to the middle sheet points from $\mathcal{E}_1$ towards $\mathcal{E}_2$. We say that the type of the triple point $\mathcal{T}$ is
\begin{itemize}
\item[$\langle0\rangle$] if the other boundary point of both $\mathcal{E}_1$ and $\mathcal{E}_2$ is a triple point,
\item[$\langle2 \rangle$] if the other boundary point of $\mathcal{E}_1$ is a triple point, while the other boundary point of $\mathcal{E}_2$ is a branch point,
\item[$\langle5\rangle$] if the other boundary point of $\mathcal{E}_2$ is a triple point, while the other boundary point of $\mathcal{E}_1$ is a branch point,
\item[$\langle25\rangle$] if the other boundary point of both $\mathcal{E}_1$ and $\mathcal{E}_2$ is a branch point.
\end{itemize}
We denote by $t_{\omega}^{\epsilon}(\lambda)$ the number of triple points of $\Delta$ with the sign $\epsilon$, type $\langle\omega\rangle$ and Alexander numbering $\lambda$. Let $t_w(\lambda)$ be the sum of signs for all triple points of type $\langle\omega\rangle$ with Alexander numbering $\lambda$. Satoh in \cite{paper_2Sat2} found the following obstruction on the projection of a surface-knot.
\begin{equation}
t_0(\lambda)+2t_2(\lambda)+t_5(\lambda)+2t_{25}(\lambda)=t_0(\lambda +1)+t_2(\lambda +1)+2t_5(\lambda +1)+2t_{25}(\lambda +1) 
\label{x}
\end{equation}
As a direct consequence of Equation (\ref{x}), we have the following lemma.
\begin{lem}[\cite{paper_2Sat3}]\label{equa}
Assume that $F$ is a surface-knot with $t(F)=2$. Let $\Delta$ be a $t$-minimal surface-knot diagram of $F$ whose triple points are $\mathcal{T}_1$ and $\mathcal{T}_2$. Then $\mathcal{T}_1$ and $\mathcal{T}_2$ are of the same type with $\epsilon(\mathcal{T}_1)=-\epsilon(\mathcal{T}_2)$ and $\lambda({\mathcal{T}_1})=\lambda(\mathcal{T}_2)$.
\end{lem} 
\begin{proof}
A proof can be found in \cite{paper_2Sat3}.
\end{proof}
\section{The algebraic intersection number of first homology elements of the torus}
Let $\mathbb{S}^{1}$ be the unit circle with the positive orientation. Let $T=\mathbb{S}^{1}\times\mathbb{S}^{1}$ be the standard torus. Assume that $l_1$ and  $l_2$ are simple closed curves in $T$ that intersect transversally at some isolated crossing points. The algebraic intersection number between  $l_1 $ and  $l_2 $ is defined as follows  
\begin{defn}
Let $p$ be a point of intersection between $l_1 $ and  $l_2 $.  The  intersection index assigned to  $p$, denoted by  $i_p ( l_1 , l_2)$,  is $+1$ if the tangent vectors to the pair $(l_1,l_2)$ form an oriented basis for the tangent plane at that point $p$
and $-1$ otherwise. Then the \textit{algebraic intersection number} between $l_1 $ and  $l_2$ , denoted by $I (l_1, l_2 )$,  is defined by the sum of the indices of the intersection points of $l_1 $ and  $l_2 $, that is 
\[I ( l_1,l_2 )= \sum_{p \in l_1 \cap l_2} i_p ( l_1, l_2 )
\]
\end{defn}
Two simple closed curves $l_1$ and $l_2$ in $T$ are said to be \textit{homologous} if $l_{1}-l_{2}$ bounds a 2-chain of the chain group of $T$. We use $[l_1]=[l_2]$ to indicate that $l_1$ and $l_2$ are homologous. Note that the algebraic intersection number depends only on the homology classes. 
Let $[l]$ be an element of the first homology group of the torus, $H_1(T)=\mathbb{Z} \times \mathbb{Z}$. In fact, $l$ is a simple closed curve in the torus which can be represented as some point $(p,q) \in \mathbb{Z} \times \mathbb{Z}$. 
\begin{thm}[\cite{cup}]
For the torus $T$, the algebraic intersection number of two simple closed curves  $(p,q)$ and $(p',q')$ is given by
\[
I \big( (p,q),(p',q')\big)=pq'-p'q
\]
\end{thm}
In the next two sections, we will present some figures in which a box in a simple closed curve means that it might be twisted or knotted.
\section{Lemmas}
Throughout this section, $F$ is assumed to be a genus-one surface-knot with $t(F)=2$. Also suppose  $\Delta$ is a $t$-minimal surface-knot diagram of $F$ whose triple points are $\mathcal{T}_1$ and $\mathcal{T}_2$. By Lemma \ref{equa},  $\mathcal{T}_1$ and $\mathcal{T}_2$ have same Alexander numbering and with opposite signs. In each lemma of this section, we show that $\Delta$ can be transformed to a diagram of $F$ with no triple points. This contradicts the assumption that $\Delta$ is a $t$-minimal diagram and so we get the result in each lemma that $t(F)=0$.

\begin{lem}\label{cancelling}
Suppose $\mathcal{E}_i$ $(i=1,2,3,4,5,6)$ is a double edge in $\Delta$ bounded by $\mathcal{T}_{1}$ and $\mathcal{T}_{2}$ such that it is of the same type at both triple points. Then $t(F)=0$.
\end{lem}
\begin{proof}
Since $\mathcal{E}_i$ $(i=1,2,3,4,5,6)$ is of the same type at both triple points, we may arrange the double edges $\mathcal{E}_i$'s so that each of $\mathcal{E}_1$ and $\mathcal{E}_2$ is the $m/t$-branch at $\mathcal{T}_{1}$ and $\mathcal{T}_2$. Let $\mathcal{E}_3$ and $\mathcal{E}_4$ be the $b/m$-branches at both triple points and let $\mathcal{E}_5$ and $\mathcal{E}_6$ denote the $b/t$-branches. We obtain three double point circles in $\Delta$, namely $C_1=\midpoint{\mathcal{E}}_1 \cup \midpoint{\mathcal{E}}_2$, $C_2=\midpoint{\mathcal{E}}_3 \cup \midpoint{\mathcal{E}}_4$ and $C_3=\midpoint{\mathcal{E}}_5 \cup \midpoint{\mathcal{E}}_6$. The double decker set and its projected image are shown in Figure \ref{can1}. If $(\mathcal{T}_{1},\mathcal{T}_{2})$ forms a $2$-cancelling pair in $\Delta$, then $\mathcal{T}_{1}$ and $\mathcal{T}_{2}$ can be eliminated by the move $R$-$2^-$ and so we have a contradiction to the assumption that  $\Delta$ is a $t$-minimal diagram.\\
Suppose on the other hand that the pair $(\mathcal{T}_{1},\mathcal{T}_{2})$ does not form a $2$-cancelling pair. In this case, we need to prove that $\Delta$ can be transformed by a finite sequence of Roseman moves into a diagram of $F$ with two triple points forming $2$-cancelling pair.  This can be done as follows.  First, we suppose that $C_i^U$ $(i=1,2,5)$ are contained in the set $\widetilde{W} \subset F$, where $\widetilde{W}$ corresponds to the set $\widetilde{X}$ of the condition (1) and we assume that (1) does not hold. \\
We establish the claim below that is necessary to prove the condition (1) \\
{\bf{Claim 1:}} There exists $(i,j) \in \{3,4\} \times \{5,6\}$ such that $\midpoint{\mathcal{E}_i^L} \cup \midpoint{\mathcal{E}_j^L}$ bounds a disk in $F$.\\
{\bf{Proof of Claim 1:}} Let  $L=\{a,b,c,d\}$ be a set of oriented closed paths in $F$ such that $a,b,c$ and $d$ are contained in a tubular neighbourhood of $\midpoint{\mathcal{E}_3^L} \cup \midpoint{\mathcal{E}_{5}^L}$, $\midpoint{\mathcal{E}_3^L} \cup \midpoint{\mathcal{E}_{6}^L}$,  $\midpoint{\mathcal{E}_4^L} \cup \midpoint{\mathcal{E}_{5}^L}$ and $\midpoint{\mathcal{E}_4^L} \cup \midpoint{\mathcal{E}_{6}^L}$, respectively  (see Figure \ref{claim}, where the elements of $L$ are denoted by dotted loops).  Suppose that for all $l \in L$,  $l$ represents a non-trivial class of $H_{1}(F)$. The elements of $L$ are pairwise disjoint by the definition. Therefore,  we have $[l_1]=[l_2]$ for distinct $l_1,l_2 \in L$.  Let the regions bounded by $a$, $b$, $c$ and $d$ be oriented as shown in Figure \ref{claim}. We obtain $[b]=[a]+[c]+[d]$. But this contradicts the fact that $[a],[b],[c]$ and $[d]$ are all homologous in $H_1(F) \cong \mathbb{Z} \oplus \mathbb{Z}$.  \\ \\  
From Claim 1, we may assume that  $\midpoint{\mathcal{E}_3^L} \cup \midpoint{\mathcal{E}_5^L}$ bounds a disk in $F$, denoted by $\widetilde{E}$. Let $N(\widetilde{E})$ be a 2-ball neighbourhood of $\widetilde{E}$ in $F$. Let $E=p(\widetilde{E})$ and $\widehat{E}=p\big(N(\widetilde{E})\big)$, where $p$ is the orthogonal projection  (see Figure \ref{can1}). In $\vert \Delta \vert$, we define $\mathcal{E}_+$ and $\mathcal{E}_- \subset \widehat{E}$ by $\mathcal{E}_+=(\mathcal{E}_5 \cup \mathcal{E}_6) \cap \widehat{E}$ and $\mathcal{E}_-=(\mathcal{E}_3 \cup \mathcal{E}_4) \cap \widehat{E}$. Let $J=[-1,1]$. Take closed neighbourhoods $N(\mathcal{E}_3^U)$ and $N(\mathcal{E}_5^U)$ in $F$ of $\mathcal{E}_3^U$ and $\mathcal{E}_5^U$, respectively such that 
      
\begin{itemize}
\item[-] $p\big(N(\mathcal{E}_3^U)\big) \cong \mathcal{E}_- \times J$, where $\mathcal{E}_- \times \{0\}= \mathcal{E}_-$, and 
\item[-] $p\big(N(\mathcal{E}_5^U)\big) \cong \mathcal{E}_+ \times J$, where $\mathcal{E}_+ \times \{0\}= \mathcal{E}_+$, and     
\item[-] $p\big(N(\mathcal{E}_3^U)\big) \cap p\big(N(\mathcal{E}_5^U)\big) \cong \partial \mathcal{E}_3 \times J = \partial \mathcal{E}_5 \times J$. 
\end{itemize}
We denote $p\big(N(\mathcal{E}_3^U)\big)$ and $p\big(N(\mathcal{E}_5^U)\big)$ by $N(\mathcal{E}_-)$ and $N(\mathcal{E}_+)$, respectively. There is a closed 3-ball neighbourhood $N(\widehat{E})$ of $\widehat{E}$ in $\mathbb{R}^3$ which is homeomorphic to $\widehat{E} \times J$, where $\widehat{E} \times \{0\}=\widehat{E}$ and contains $N(\mathcal{E}_-) \cup N(\mathcal{E}_+)$.
The disk $E \times \{i\}$ $(i=-1/2,1/2)$  is an embedded disk in $\mathbb{R}^3$ that is parallel to the embedded disk $E=E \times \{0\}$. Since $t(\Delta)=2$,  the interior of $E$ does not contain neither branch points nor triple points. In particular, the interior of $E$ may contains some simple closed double curves. By Lemma \ref{50}, we may assume that the interior of $E$ does not meet the projection $\vert \Delta \vert$ and so the interior of $E \times \{i\}$ $(i=-1/2,1/2)$ does not. Therefore for $\epsilon >0$, the pair $\big(\widehat{E} \times [-1,\epsilon], \mathcal{E}_+ \times [-1,\epsilon] \cup \mathcal{E}_- \times [-1,\epsilon] \cup E \times \{-1/2\}\big)$ is homeomorphic to the model of a descendent disk. Also, the pair $\big(\widehat{E} \times [\epsilon,1], \mathcal{E}_+ \times [\epsilon,1] \cup \mathcal{E}_- \times [\epsilon,1] \cup E \times \{1/2\}\big)$ is homeomorphic to the model of a descendent disk. In particular, the disk $E \times \{i\}$ $(i=-1/2,1/2)$ is a descendenent disk and therefore, we can apply the Roseman move $R$-$7$ along it. A new surface diagram is obtained in which the condition (1) of $2$-cancelling pair holds; that is we have new double decker sets $\mathcal{E}_1^{U'}$ and $\mathcal{E}_2^{U'}$ with the closure of $\mathcal{E}_i^{U'} \cup \mathcal{E}_5^{U'}$ $(i=1,2)$ bounds a disk in $F$ with no double decker set in its interior (see Figure \ref{*}). Now we can go through a similar  procedure that we did to the disk $\widetilde{E}$ explained above to any of $\midpoint{\mathcal{E}_1^{U'}} \cup \midpoint{\mathcal{E}_5^{U'}}$ or $\midpoint{\mathcal{E}_2^{U'}} \cup \midpoint{\mathcal{E}_5^{U'}}$. As a result, we obtain a surface-knot diagram of $F$ with two triple points which form a 2-cancelling pair.  
\end{proof}

\begin{figure}[H]
  \centering
  \captionsetup{font=scriptsize}  
  \subfloat[]{\includegraphics[width=0.46\textwidth]{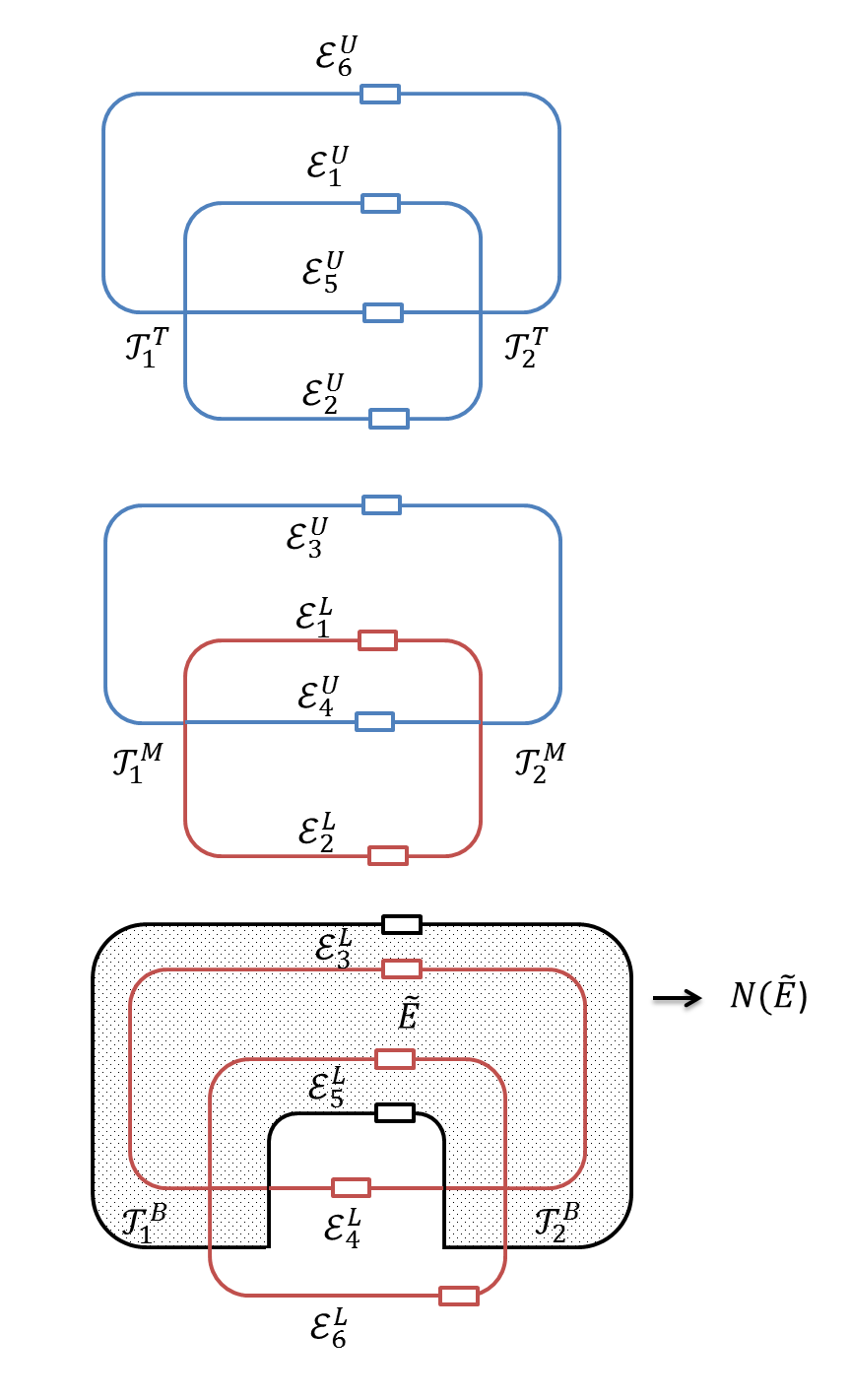}\label{fig:f2}}
   \hfill
   \subfloat[]{\includegraphics[width=0.54\textwidth]{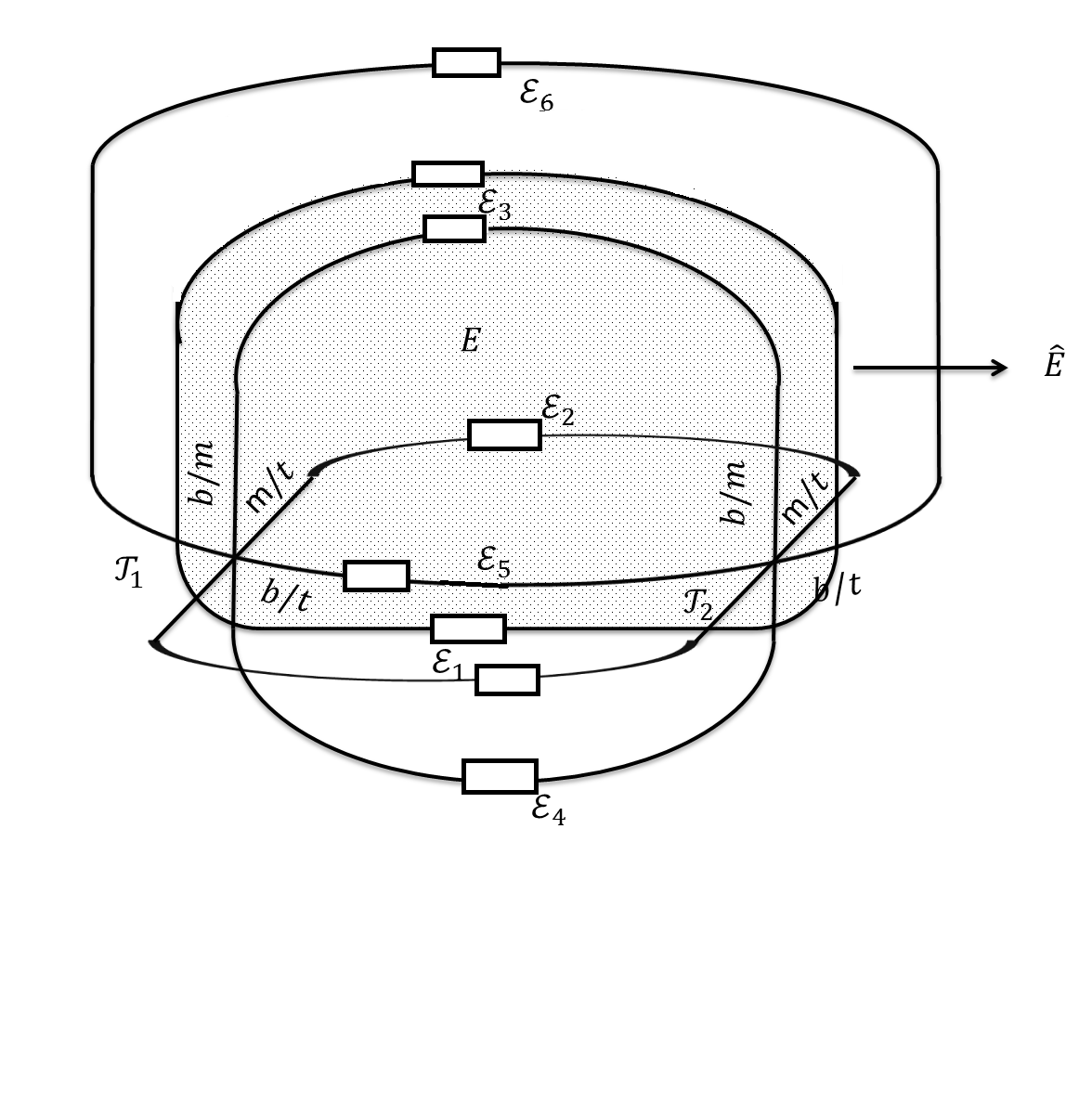}\label{fig:f1}}
  \caption{Lemma \ref{cancelling}  (a) The pre-image of the closure of the double edges  \quad (b) the connections of the double edges in the projection}
  \label{can1}
\end{figure}
\begin{figure}[H]
\centering
\captionsetup{font=scriptsize}      
\mbox{\includegraphics[scale=0.4]{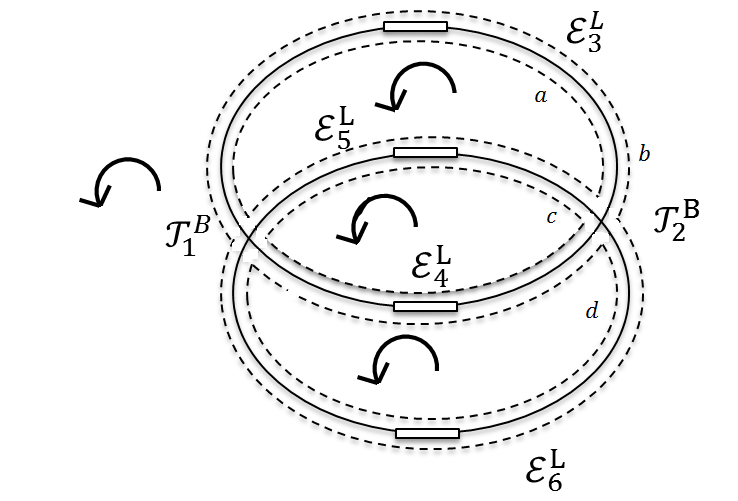}}
 \caption{}
 \label{claim}
\end{figure}
\begin{figure}[H]
\centering
\captionsetup{font=scriptsize}      
\mbox{\includegraphics[scale=0.6]{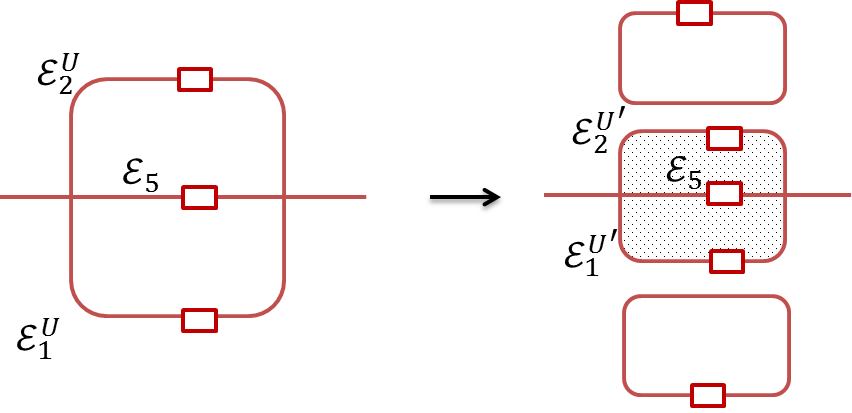}}
 \caption{}
\label{*}
\end{figure}

\begin{lem}\label{Pinch}
Suppose both $\mathcal{E}_1$ and $\mathcal{E}_2$ are  double loops based at $\mathcal{T}_1$ in $\Delta$  such that $\mathcal{E}_i$ $(i=1,2)$ has a $b/t$- and $m/t$-branch at $\mathcal{T}_1$. 
If the $b/m$-branches at both triple points of $\Delta$ are joined, then $F$ satisfies $t(F)=0$. 
\end{lem} 
\begin{proof}
Let $C_1=\midpoint{\mathcal{E}}_1 \cup \midpoint{\mathcal{E}}_2$ be a double point circle in $\Delta$ such that $\mathcal{E}_i$ $(i=1,2)$ is a $b/t$- and $m/t$-branch at $\mathcal{T}_1$.  Let $C_2=\midpoint{\mathcal{E}}_3 \cup \midpoint{\mathcal{E}}_4$ be a double point circle in $\Delta$ such that each of $\mathcal{E}_i$ $(i=3,4)$ is a $b/m$-branch at $\mathcal{T}_1$ and $\mathcal{T}_2 $.  Since $t(\Delta)=2$, we have the double point circle $C_3=\midpoint{\mathcal{E}}_5 \cup \midpoint{\mathcal{E}}_6$ in $\Delta$ such that $\mathcal{E}_i$ $(i=5,6)$ is a $b/t$- and $m/t$-branch at $\mathcal{T}_2$. The double decker set and its projected image in 3-space are depicted in Figure \ref{L2}. Let $C_1^L$ and $C_2^L$ be the lower decker curves of $C_1$ and $C_2$ in $F$, respectively. $C_1^L$ and $C_2^L$ intersect at only one crossing point; that is $\mathcal{T}_1^B$. This implies that $[C_1^L]$ and $[C_2^L]$ are distinct non-trivial elements in $H_1(F)$. Suppose for the sake of contradiction that there exists $i \in \{1,2\}$ such that $\midpoint{\mathcal{E}_i^U}$ does not bound a disk in $F$. Then $[\midpoint{\mathcal{E}_i^U}]$ represents a non-trivial element in homology $H_1(F)$ which is distinct from $[C_1^L]$ or from $[C_2^L]$. Therefore, $\midpoint{\mathcal{E}_i^U}$ must intersect $C_1^L$ or $C_2^L$, a contradiction. We obtain that any of $\midpoint{\mathcal{E}_i^U}$ $(i=1,2)$ bounds a disk in $F$. Suppose $\widetilde{P}$ is the disk bounded  by $\midpoint{\mathcal{E}_1^U}$ and let $P=p(\widetilde{P})$. For $\epsilon >0$, assume that $P\times \{\epsilon\}$ is an embedded disk in $\mathbb{R}^3$ that is parallel to $P$ and transversal to $\Delta$ along $\partial (P \times \{\epsilon\})$. Let $\mathcal{M}_2$, $\mathcal{M}_3$ and $\mathcal{S} \subset \vert \Delta \vert$ denote the set of double points, triple points and branch points, respectively. Because $t(\Delta)=2$, $\big($int$(P)\big) \cap (\mathcal{M}_3 \cup \mathcal{S}) = \emptyset$. In particular, the interior of $P$ might contains some simple closed double curves. By Lemma \ref{51}, we can suppose  that the interior of $P \times \{\epsilon\}$ does not meet $\vert \Delta \vert$. In fact, $P \times \{\epsilon\}$ has a closed neighbourhood $N(P \times \{\epsilon\})$ in $\mathbb{R}^3$ such that the pair $( N(P \times \{\epsilon\}), N(P \times \{\epsilon\}) \cap \vert \Delta \vert)$ is homeomorphic to the model of a pinch disk. Now, we apply the move $R$-$5^+$ to create a pair of branch points and then moving one of the branch points along the $m/t$-branch at $\mathcal{T}_1$ so that $\mathcal{T}_1$ is eliminated. This completes the proof.
\end{proof}
\begin{figure}[H]
  \centering
  \captionsetup{font=scriptsize}  
   \subfloat[]{\includegraphics[width=0.5\textwidth]{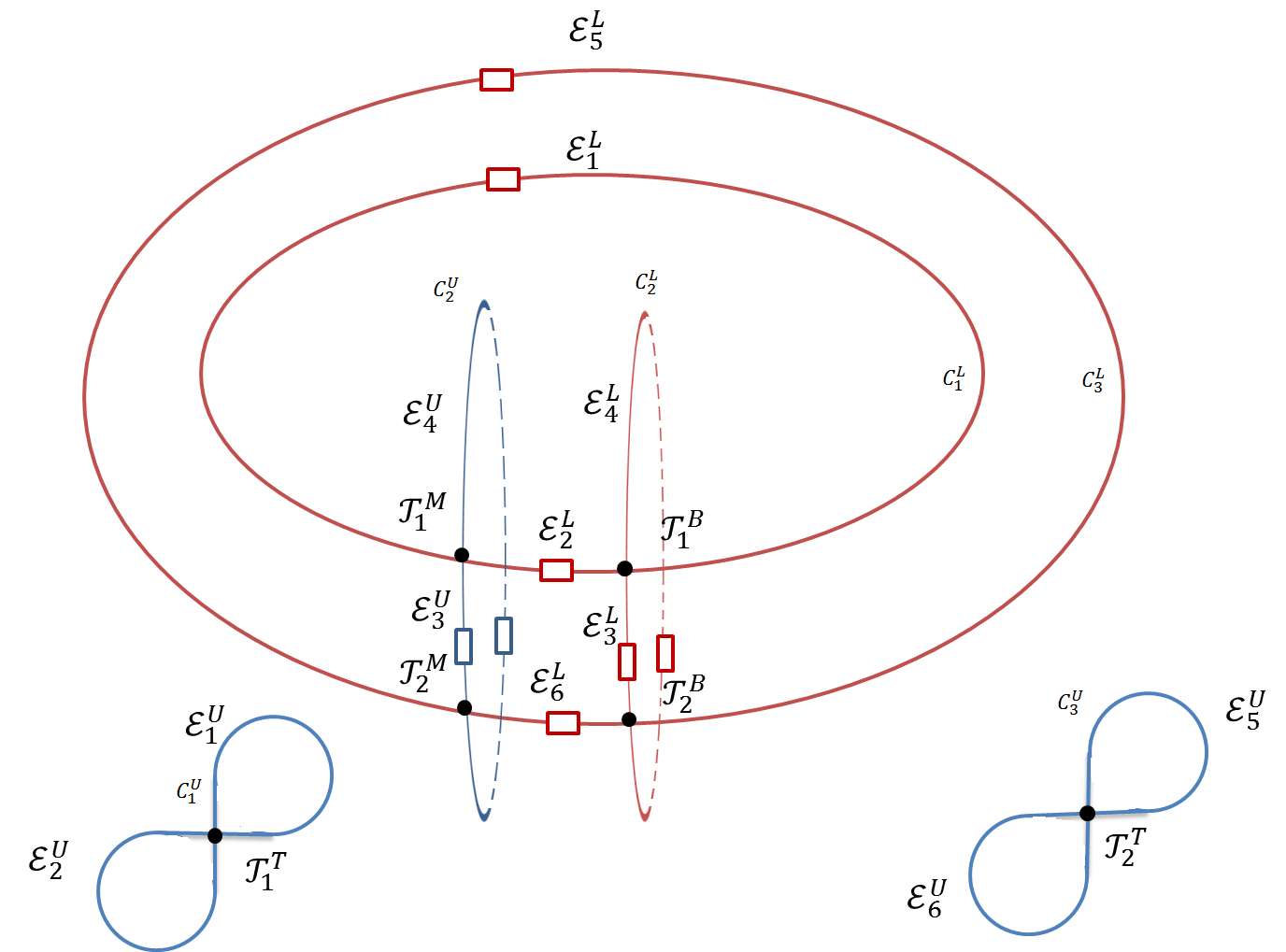}\label{fig:f2}}
   \hfill
  \subfloat[]{\includegraphics[width=0.5\textwidth]{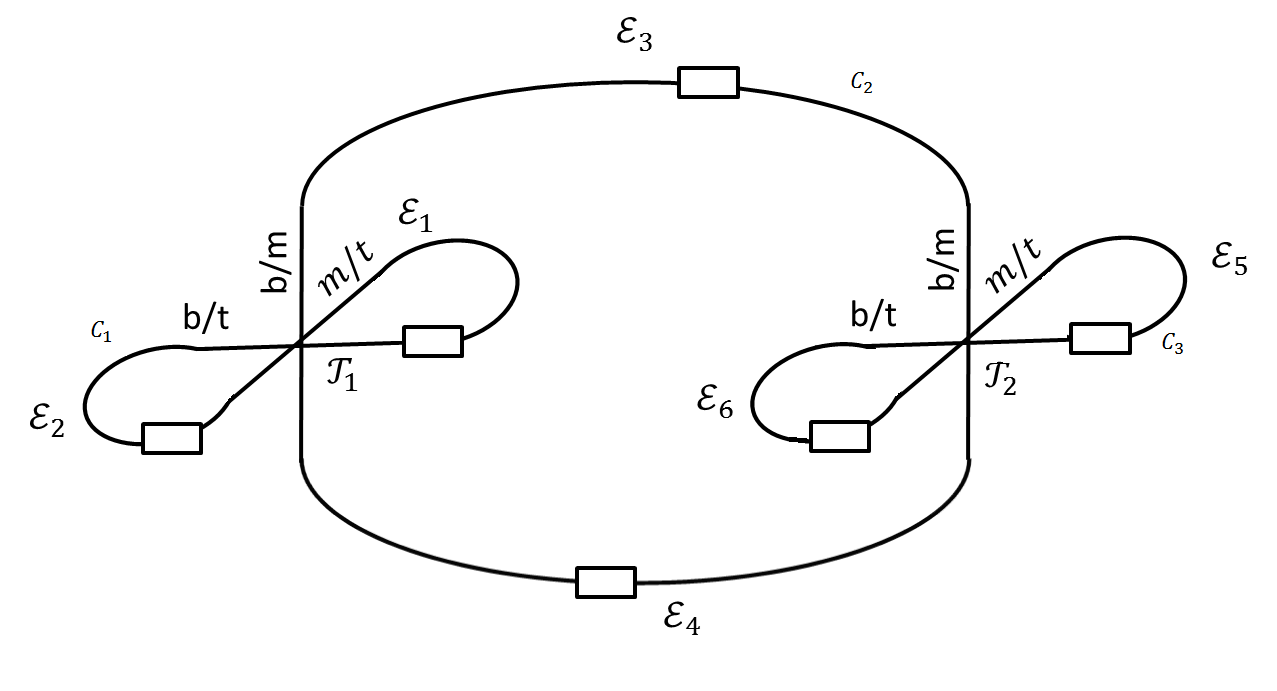}\label{fig:f1}}
 \caption{Lemma \ref{Pinch}  (a) The pre-image of the closure of the double edges  \quad (b) the connections of the double edges in the projection}
  \label{L2}
\end{figure}

\begin{lem}\label{imp}
Suppose that there exists a double point circle $C=\midpoint{\mathcal{E}}_1 \cup \midpoint{\mathcal{E}}_2$ such that $\mathcal{E}_i$ $(i=1,2)$ is a $b/m$-branch at both triple points of $\Delta$.
Then, the surface-knot $F$ satisfies $t(F)=0$.
\end{lem}
\begin{proof}
Let $\mathcal{E}_3$ and $\mathcal{E}_4$ be the $m/t$-branches at $\mathcal{T}_1$.  If the other boundary point of any of $\mathcal{E}_i$ $(i=3,4)$ is a branch point, then the result follows from Lemma \ref{R6}. Suppose on the other hand  that the other boundary point of any of $\mathcal{E}_i$ $(i=3,4)$ is a triple point.  Since $\mathcal{M}_3=\{\mathcal{T}_1,\mathcal{T}_2\}$,  we have to consider the following cases
\begin{itemize}
\item[Case 1.] The other boundary point of $\mathcal{E}_i$ $(i=3,4)$ is $\mathcal{T}_2$. From the Alexander numbering assigned to the eight regions around the triple points $\mathcal{T}_1$ and $\mathcal{T}_2$, it follows that each of $\mathcal{E}_3$ and $\mathcal{E}_4$ is a $m/t$-branch at $\mathcal{T}_2$. By Satoh's identity (\ref{x}), $\mathcal{T}_1$ and $\mathcal{T}_2$ are of the same type. Assume that both triple points are of type $<2>$. Let $\mathcal{E}_5$ be the $b/t$-branch at $\mathcal{T}_1$ such that the orientation normal to the middle sheet points towards $\mathcal{E}_5$. The other endpoint of $\mathcal{E}_5$ is a branch point. Also let  $\mathcal{E}_6$ be the $b/t$-branch at $\mathcal{T}_2$ such that the orientation normal to the middle sheet points towards $\mathcal{E}_6$. The other endpoint of $\mathcal{E}_6$ is a branch point. Since $\mathcal{T}_1$ and $\mathcal{T}_2$ have same Alexander numbering and opposite signs,  we have $\lambda(\mathcal{B}_1)=\lambda(\mathcal{B}_2)$ and $\epsilon(\mathcal{B}_1)=-\epsilon(\mathcal{B}_2)$.  Because $\mathcal{E}_1$ is a $b/m$-branch at both triple points of $\Delta$, there is an embedded arc $\gamma$ in $\Delta$ which misses the double decker set except the boundary and has a neighbourhood as shown in Figure \ref{branch}. Hence we can apply the Roseman move $R$-$5^-$ to $\Delta$ so that we obtain a new surface-knot diagram of $F$ which has no branch points. We can apply the same operation if the triple points of $\Delta$ are of type $<5>$ or $<25>$. Hence we may assume that $\Delta$ has no branch points. Now  $t(F)=0$ follows from Lemma \ref{cancelling}.\\

\item[Case 2.]   The other boundary point of $\mathcal{E}_i$ $(i=3,4)$ is $\mathcal{T}_1$. In this case $\mathcal{E}_i$ $(i=3,4)$ is a double loop based at $\mathcal{T}_1$ with the property that it is a $b/t$- and $m/t$- branch at $\mathcal{T}_1$ (For if $\mathcal{E}_3$ and $\mathcal{E}_4$ coincide, we obtain a double point circle with single triple point, contradicts Lemma \ref{even}) . Now Lemma \ref{Pinch} implies $t(F)=0$.\\
\item[Case 3.]  The other boundary point of $\mathcal{E}_3$ is $\mathcal{T}_1$ and  the other boundary point of $\mathcal{E}_4$ is $\mathcal{T}_2$. In this case, $\mathcal{E}_3$ is a double loop based at $\mathcal{T}_1$ such that it is a $b/t$- and $m/t$-branch at $\mathcal{T}_1$ and $\mathcal{E}_4$ is a $m/t$-branch at both $\mathcal{T}_1$ and $\mathcal{T}_2$. Let $C_1$ be a double point circle in $\Delta$ containing $\mathcal{E}_i$ $(i=3,4)$. We may assume as in the first case that both triple points $\mathcal{T}_1$ and $\mathcal{T}_2$ are of type $<0>$. From  Lemma \ref{even}, we obtain that $C_1$ contains two other double edges $\mathcal{E}_5$ and $\mathcal{E}_6$ with the following properties:  The double edge $\mathcal{E}_5$ is a double loop based at $\mathcal{T}_2$ such that it is a $b/t$- and $m/t$-branch at $\mathcal{T}_2$ and the double edge $\mathcal{E}_6$ is a $b/t$-branch at $\mathcal{T}_1$ and $\mathcal{T}_2$.  The double decker set is depicted in  Figure \ref{33} (a) and its projected image in 3-space is shown in  Figure \ref{33} (b). Consider the upper decker curve $C_1^U$ of $C_1$. If $\midpoint{\mathcal{E}_3^U}$ or $\midpoint{\mathcal{E}_5^U}$ is on the boundary of a disk in $F$, then $t(F)=0$ follows from the proof of Lemma \ref{Pinch}. Suppose on the other hand that neither  $\midpoint{\mathcal{E}_3^U}$ nor $\midpoint{\mathcal{E}_5^U}$  bounds a disk in $F$. We follow a similar argument of Claim 1 to show that the region bounded by  $\midpoint{\mathcal{E}_4^U} \cup \midpoint{\mathcal{E}_6^U}$ must be homeomorphic to a disk, denoted by $\widetilde{M}$, in $F$.  Let $M=p(\widetilde{M})$. We can go through the similar operations that we did in Lemma \ref{cancelling} to show that there exists a descendent disk in $\mathbb{R}^3$ that is parallel to $M$ and involves two double points on $\mathcal{E}_1$. We apply the move $R$-$7$ along this descendent disk and as a result, we obtain new double decker set $\mathcal{E}_1^{L'}$ with the closure of $\mathcal{E}_1^{L'} \cup \mathcal{E}_6^L$ bounds a disk in $F$, denoted by $\widetilde{M_1}$. Then by a similar way, we can verify the existence of a descendent disk in $\mathbb{R}^3$ that is parallel to $p(\widetilde{M_1})$ and involves a point on $\mathcal{E}_3$ and a point on $\mathcal{E}_5$. By applying the move $R$-$7$ along this descendent disk, the hypothesis of Lemma \ref{cancelling} is satisfied and thus we get the conclusion.
\end{itemize}
\end{proof}

\begin{figure}[H]
\centering
\captionsetup{font=scriptsize}      
\mbox{\includegraphics[scale=0.4]{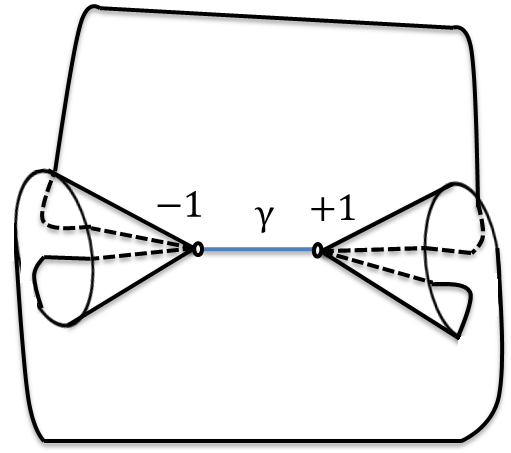}}
 \caption{}
\label{branch}
\end{figure} 
\begin{figure}[H]
  \centering
  
 \captionsetup{font=scriptsize} 
  \subfloat[]{\includegraphics[width=0.55\textwidth]{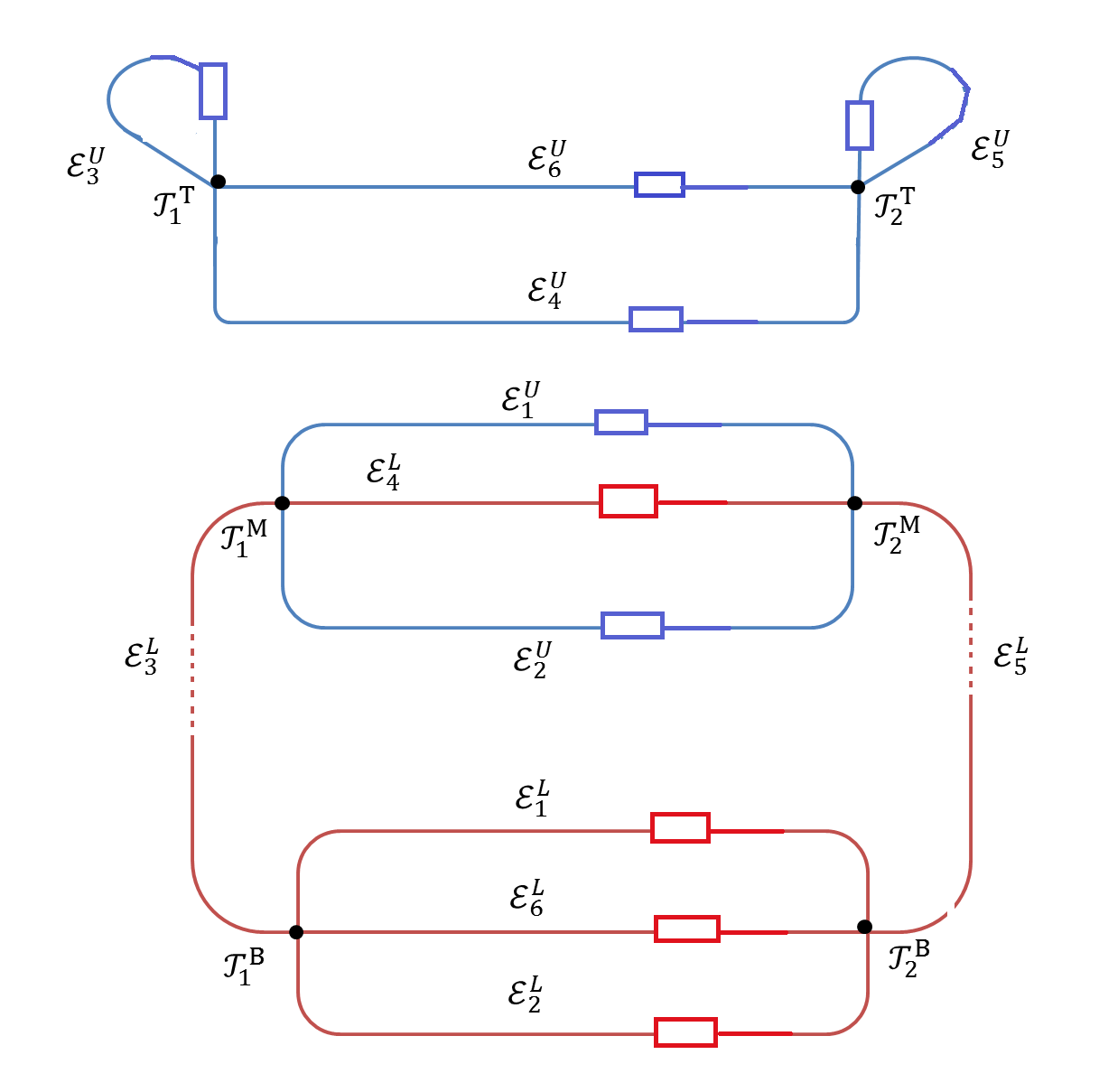}\label{fig:f2}}
   \hfill
  \subfloat[]{\includegraphics[width=0.45\textwidth]{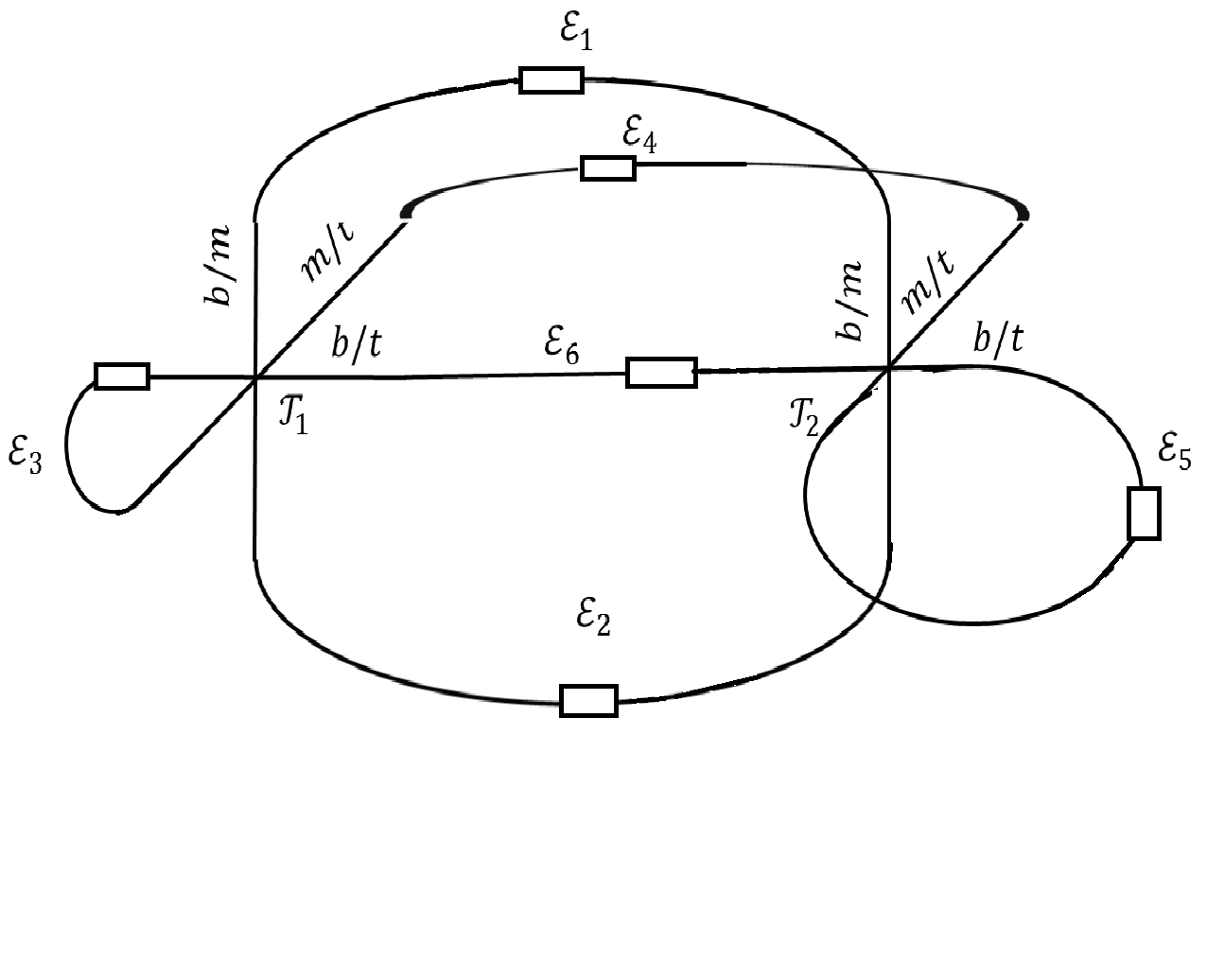}\label{fig:f1}}
  \caption{Case 3 of Lemma \ref{imp}   (a) The pre-image of the closure of the double edges  \quad (b) the connections of the double edges in the projection}
  \label{33}
\end{figure}

\section{Proof of Theorem \ref{xy}}
\begin{proof}
Assume that there is a genus-one surface-knot $F$ satisfying $t(F)=2$ . Let $\Delta$ be a $t$-minimal surface-knot diagram of $F$ with the triple points $\mathcal{T}_1$ and $\mathcal{T}_2$. By Satoh's identity (\ref{x}), we assume that $\epsilon(\mathcal{T}_1) = -\epsilon(\mathcal{T}_2)$ and that  $\lambda(\mathcal{T}_1)=\lambda(\mathcal{T}_2)$.  Note that the other endpoint of a $b/t$-branch at $\mathcal{T}_1$ or $\mathcal{T}_2$  might be a branch point.  In particular, $\mathcal{T}_1$ and $\mathcal{T}_2$ are of the same type by Satoh's identity (\ref{x}). Let $\mathcal{E}_i$ $(i=1,2)$  be  a double edge in $\Delta$ such that $\mathcal{E}_i$ $(i=1,2)$ is a $b/m$-branch at both $\mathcal{T}_1$ and $\mathcal{T}_2$.  From Lemma \ref{imp}, we may assume that there is no double point circle $C$ in $\Delta$ such that $C=\midpoint{\mathcal{E}_1} \cup \midpoint{\mathcal{E}_2}$.
There are the six cases by the following (i) the orientation of the double branches incident to the triple points; (ii)  the Alexander numbering assigned to the set of complementary connected regions $\mathbb{R}^3\setminus \vert \Delta \vert$; (iii) Lemma \ref{even}; and (iv) Remark \ref{conn}.
We show that for some cases, $\Delta$ is not a $t$-minimal and in some other cases that there is no such a diagram. So in both cases, we get a contradiction. Since there is no triple points other than $\mathcal{T}_1$ and $\mathcal{T}_2$, these six cases are sufficient for the proof. There are figures illustrating the connection of the double edges at the end of each case.
\begin{itemize}
\item[Case 1.] There are two double point circles $C_1$ and $C_2$:
\[   \left\{
\begin{array}{ll}
      C_1: \midpoint{\mathcal{E}_1} \cup \midpoint{\mathcal{E}_2} \cup \midpoint{\mathcal{E}_3} \cup \midpoint{\mathcal{E}_4}; & \\
    C_2: \midpoint{\mathcal{E}_5} \cup \midpoint{\mathcal{E}_6}.&  \\
   
\end{array} \
\right. \]  

where $\mathcal{E}_4$ (resp. $\mathcal{E}_2$) is a $b/m$- (resp. $m/t$)-branch at both $\mathcal{T}_1$ and $\mathcal{T}_2$. The double edge $\mathcal{E}_3$ is a $b/m$-branch at $\mathcal{T}_1$ and a $m/t$-branch at $\mathcal{T}_2$ while $\mathcal{E}_1$ is a $m/t$-branch at $\mathcal{T}_1$ and a $b/m$-branch at $\mathcal{T}_2$. The double edge $\mathcal{E}_i$ $(i=5,6)$ is a $b/t$-branch at both triple points of $\Delta$. \\   
The upper decker curve $C_1^U$ is a closed path in $F$. Suppose $C_1^U$ bounds by a disk in $F$. Then the closure of $\mathcal{E}_2^U \cup \mathcal{E}_5^U$  bounds a disk in $F$. We can show by a similar way of the proof of Lemma \ref{cancelling} that there is a descendent disk in the closure of $\mathbb{R}^3 \setminus  \vert \Delta \vert $ with its boundary contains a double point on $\mathcal{E}_1$ and a double point on $\mathcal{E}_3$. By applying the $R$-$7$ move along this descendent disk, the assumption of Lemma \ref{cancelling} is satisfied and therefore $t(F)=0$. This contradicts the assumption that $\Delta$ is a $t$-minimal. Similar proof is considered if $C_1^L$ bounds a disk in $F$. On the other hand, suppose that neither $C_1^U$ nor $C_1^L$ is homotopic to a trivial disk in $F$. Then since the oriented intersection number $I(C_1^U,C_1^L)=0$, we have $[C_1^U]=[C_1^L]$ in $H_1(F)$. We can show by a similar proof of Claim 1 that  the region bounded by $\midpoint{\mathcal{E}_4^U} \cup \midpoint{\mathcal{E}_2^L}$ is a disk in $F$, denoted by $\widetilde{R}$. In particular,  there exists a descendent disk $S$ in one of the complementary open regions of the projection that is parallel to $p(\widetilde{R})$ such that the boundary of $S$  contains two double points on $\mathcal{E}_5$.  Apply the move $R$-$7$ along $S$. As a consequence, we obtain a new double decker set $\mathcal{E}_5^{U'}$ such that the closure of $\mathcal{E}_2^U \cup \mathcal{E}_5^{U'}$  is a disk in $F$, denoted by $\widetilde{M}$. Now by following the similar argument of Lemma \ref{cancelling}, we can find a descendent disk in $\mathbb{R}^3$ that is parallel to $p(\widetilde{M})$. By applying the move $R$-$7$ to the resulting desecendent disk, the diagram $\Delta$ is transformed to a diagram satisfying the assumption of Lemma \ref{cancelling} and thus we get a contradiction.

\item[Case 2.] There are a double point circle $C_1$ and a double point interval $C_2$:
\[   \left\{
\begin{array}{ll}
      C_1: \midpoint{\mathcal{E}_1} \cup \midpoint{\mathcal{E}_2} \cup \midpoint{\mathcal{E}_3} \cup \midpoint{\mathcal{E}_4}; & \\
    C_2: \midpoint{\mathcal{E}_5} \cup \midpoint{\mathcal{E}_6} \cup \midpoint{\mathcal{E}_7}.&  \\
   \end{array} \
\right. \]  
such that $\mathcal{E}_i$ $(i=1,2,3,4)$ is as described in the previous case. The boundary points of the double edge $\mathcal{E}_5$  (resp. $\mathcal{E}_7$) are the triple point $\mathcal{T}_2$ (resp. $\mathcal{T}_1$ ) and a branch point. The double edge $\mathcal{E}_6$ is a $b/t$-branch at both $\mathcal{T}_1$ and $\mathcal{T}_2$. \\ Because $\mathcal{E}_1$ joins the $b/m$-branches at both triple points, we can connect the two branch points by a simple arc that misses the singularity set of the projection except the boundary. Therefore, this case is reduced to the previous case. 

\item[Case 3.] There are a double point circle $C_1$ and two double point interval $C_2$ and $C_3$:
\[   \left\{
\begin{array}{ll}
      C_1: \midpoint{\mathcal{E}_1} \cup \midpoint{\mathcal{E}_2} \cup \midpoint{\mathcal{E}_3} \cup \midpoint{\mathcal{E}_4} ;& \\
    C_2: \midpoint{\mathcal{E}_5} \cup \midpoint{\mathcal{E}_6} ;&  \\
    C_3:\midpoint{\mathcal{E}_7} \cup \midpoint{\mathcal{E}_8}
   
\end{array} \
\right. \]  
such that $\mathcal{E}_i$ $(i=1,2,3,4)$ is as described in Case 1. The boundary points of the double edge $\mathcal{E}_5$ and $\mathcal{E}_6$  (resp. $\mathcal{E}_7$ and $\mathcal{E}_8$) are the triple point $\mathcal{T}_2$ (resp. $\mathcal{T}_1$ ) and a branch point.\\ The proof of this case is analogous to the proof of the previous case.

\item[Case 4.] There are two double point circles $C_1$ and $C_2$:
  \[   \left\{
\begin{array}{ll}
      C_1: \midpoint{\mathcal{E}_1} \cup \midpoint{\mathcal{E}_2};& \\
    C_2: \midpoint{\mathcal{E}_3} \cup \midpoint{\mathcal{E}_4}.&  \\
   
\end{array} \
\right. \]  
where each of $\mathcal{E}_i$ $(i=1,2)$ is a $b/m$-branch at $\mathcal{T}_{1}$ and a $m/t$-branch at $\mathcal{T}_{2}$ and  $\mathcal{E}_i$ $(i=3,4)$ is a $m/t$-branch at $\mathcal{T}_{1}$ and a $b/m$-branch at $\mathcal{T}_{2}$. Assume $C_i^U$ is the upper decker curve of $C_i$ $(i=1,2)$ and $C_i^L$ is the lower decker curve of $C_i$ $(i=1,2)$ in $F$. In particular, $C_2^U$ intersects $C_1^L$ at one crossing point which is $\mathcal{T}_{2}^M$ . This implies that each of $[C_2^U]$ and $[C_1^L]$ is homotopic to a non-trivial element in $H_1(F)$ and indeed they represent distinct elements in $H_1(F)$. Similarly, $C_1^U \cap C_2^L =\{ \mathcal{T}_1^M\}$ in $F$. Therefore, $[C_1^U]$ represents a non-trivial element in homology. Now any of $C_2^U$ and $C_1^L$ does not meet $C_1^U$, which in turn gives that $[C_2^U]$, $[C_1^L]$ and $[C_1^U]$  are homologous in $H_1(F)$. But $[C_1^L]$ is not homologous to $[C_2^U]$. This is a contradiction. 

\item[Case 5.] There are a double point circle $C_1$ and double point circle $C_2$ with two loops:
\[   \left\{
\begin{array}{ll}
      C_1: \midpoint{\mathcal{E}_1} \cup \midpoint{\mathcal{E}_2};& \\
    C_2: \midpoint{\mathcal{E}_3} \cup \midpoint{\mathcal{E}_4} \cup \midpoint{\mathcal{E}_5} \cup \midpoint{\mathcal{E}_6}. &  \\
   
\end{array} \
\right. \]  
where each of $\mathcal{E}_i$ $(i=1,2)$ is a $b/m$-branch at $\mathcal{T}_{1}$ and a $m/t$-branch at $\mathcal{T}_{2}$ while the double edge $\mathcal{E}_3$ is a $m/t$-branch at $\mathcal{T}_{1}$ and a $b/m$-branch at $\mathcal{T}_{2}$. The double edge $\mathcal{E}_4$ (resp. $\mathcal{E}_6$) is a double loop based at $\mathcal{T}_1$ (resp. $\mathcal{T}_2$) such that it is a $b/t$- and a $m/t$- (resp. $b/t$ and $b/m$)- branch. The double edge $\mathcal{E}_5$ is a $b/t$-branch at both triple points. \\
We have $C_1^U \cap (C_2^U \setminus \mathcal{E}_4^U) =\{\mathcal{T}_2^T\}$, $C_1^L \cap (C_2^U \setminus \mathcal{E}_4^U) =\{\mathcal{T}_2^M\}$ and $C_1^U \cap (C_2^L \setminus \mathcal{E}_6^L)=\{\mathcal{T}_1^M\}$. Therefore, $[C_2^U \setminus \mathcal{E}_4^U] =[C_2^L \setminus \mathcal{E}_6^L]$ in $H_1(F)$ and they represent a generator of the first homology group of $F$. The other generator is represented by $[C_1^U]=[C_1^L]$. Now $\midpoint{\mathcal{E}_4^U}$ is a closed path in the torus, denoted by $l_1$.  Suppose that $l_1$ is not spanned by a disk in $F$. It is not difficult to see that either $l_1$ must intersect transversally one of the generators. But $\Delta$ has only two triple points. This is a contradiction. Thus, $l_1$  must be on the boundary of disks in $F$. By following a similar transformation applied in the proof of Lemma \ref{Pinch}, we eliminate the two triple points and this contradicts the assumption that $\Delta$ is a $t$-minimal. 

\item[Case 6.] There are a double point circle $C_1$ and double point interval $C_2$ with two loops:
\[   \left\{
\begin{array}{ll}
      C_1: \midpoint{\mathcal{E}_1} \cup \midpoint{\mathcal{E}_2};& \\
    C_2: \midpoint{\mathcal{E}_3} \cup \midpoint{\mathcal{E}_4} \cup \midpoint{\mathcal{E}_5} \cup \midpoint{\mathcal{E}_6}  \cup \midpoint{\mathcal{E}_7}.  &  \\
   
\end{array} \
\right. \]
such that each of $\mathcal{E}_i$ $(i=1,2)$ is a $b/m$-branch at $\mathcal{T}_{1}$ and a $m/t$-branch at $\mathcal{T}_{2}$ while the double edge $\mathcal{E}_5$ is a $m/t$-branch at $\mathcal{T}_{1}$ and a $b/m$-branch at $\mathcal{T}_{2}$ . The double edge $\mathcal{E}_4$ (resp. $\mathcal{E}_6$) is a double loop based at $\mathcal{T}_1$ (resp. $\mathcal{T}_2$) such that it is a $b/t$- and a $m/t$- (resp. $b/t$ and $b/m$)- branch. The end points of the double edge $\mathcal{E}_3$ (resp. $\mathcal{E}_7$) is the triple point $\mathcal{T}_1$ (resp. $\mathcal{T}_2$) and a branch point.  \\
 Suppose that neither  $\midpoint{\mathcal{E}_4^U}$ nor $\midpoint{\mathcal{E}_6^L}$ bounds a disk in $F$. Since $\midpoint{\mathcal{E}_4^U} \cap \midpoint{\mathcal{E}_6^L}=\emptyset $,   $[\midpoint{\mathcal{E}_4^U}] = [\midpoint{\mathcal{E}_6^L}] $ in $H_1(F)$.  Let $\alpha$ be a closed path in $F$ that is defined by $\big(C_2^U \setminus \mathcal{E}_4^U\big) \cup \big(C_2^L \setminus \mathcal{E}_6^L\big)$. In particular, $\alpha$  defines a graph, $G$, in $F$ with two vertices, $a_1$ and $a_2$,  and two edges, $e_1$ and $e_2$,  as shown in  Figure \ref{cellu} such that
\begin{itemize}
\item[-]   $\mathcal{E}_5^U \cup \midpoint{\mathcal{E}_6^U} \cup \midpoint{\mathcal{E}_7^U} \cup (\midpoint{\mathcal{E}_7^L}\setminus{\mathcal{T}_{2}^B})=e_1 $, 
\item[-]   $(\midpoint{\mathcal{E}_3^U} \setminus{\mathcal{T}_{1}^T}) \cup \midpoint{\mathcal{E}_3^L} \cup \midpoint{\mathcal{E}_4^L} \cup \mathcal{E}_5^L=e_2$,
\item[-] $\mathcal{T}_{1}^T=a_1$ and $\mathcal{T}_{2}^B=a_2$. 
\end{itemize}
Suppose for the sake of contradiction that $\alpha$ does not bound a disk in $F$.  Since $\alpha \cap \midpoint{\mathcal{E}_4^U}=\{\mathcal{T}_1^T\}$, $[\mathcal{E}_4^U]$ and $ [\alpha]$ represent two distinct generators of $H_1(F)$.  Note that  $C_1^L$   intersects transversally $e_1$ at a single crossing point and intersects $e_2$ at an exactly one crossing point. From this notation, it is easy to see that  $[C_1^L]=[\midpoint{\mathcal{E}_4^U}]$  in $H_1(F)$ (By a similar argument we show that $[C_1^U]=[\midpoint{\mathcal{E}_4^U}]$  in $H_1(F)$).  Therefore, $[C_1^L] \neq [\alpha]$ in $H_1(F)$. But the intersection number $I(C_1^L,\alpha)=0$, a contradiction.  We obtain that $\alpha$ bounds a disk in $F$, i.e. $G$ is a planer graph in $F$. We get
\begin{equation}
[C_1^U] =[C_1^L]=[\midpoint{\mathcal{E}_4^U}]=[\midpoint{\mathcal{E}_6^L}] \quad \text{in} \quad  H_1(F)
 \label{cont}
 \end{equation}
There is an annulus $A$ on $F$ bounded by two parallel curves $C_1^U$ and $C_1^L$ such that $p(A)$ is homeomorphic to a torus, where $p$ is the orthogonal projection. Let $\midpoint{\mathcal{E}_4^U} \times [-1,1]$ be a neighbourhood of $\midpoint{\mathcal{E}_4^U}$ in $F$, where $\midpoint{\mathcal{E}_4^U} \times \{0\} \cong\midpoint{\mathcal{E}_4^U}$. In $\Delta$, $p\big(\midpoint{\mathcal{E}_4^U} \times [-1,1]\big)$ intersects $p(A)$ transversally at the double edge $\mathcal{E}_4$ and passes through the double edge $C_1$  so that we have the triple point $\mathcal{T}_1$.  In $p(A)$, the double loop $\midpoint{\mathcal{E}_4}$ meets $C_1$ at $\mathcal{T}_1$ which implies that 
\begin{equation}\label{123}
[C_1] \neq  [\midpoint{\mathcal{E}_4}] \quad \text{in} \quad H_1(p(A))
\end{equation} 
Obviously, there is a homotopy function $g_t:[0,1] \times [0,1] \rightarrow \Delta$ with $h([0,1]\times \{0\})=\midpoint{\mathcal{E}_4}$ and $h([0,1]\times \{1\})=p\big(\midpoint{\mathcal{E}_4^U} \times \{1\}\big)$.

Assume that $C_1^L \times [-1,1]$ is a neighbourhood of $C_1^L$ in $F$ such that $C_1^L \times \{0\} \cong C_1^L$ and $p\big(C_1^L \times [-1,0]\big)$ is contained in $p(A)$. Obviously, there is a homotopy function $h_t:[0,1] \times [0,1] \rightarrow \Delta$ with $h([0,1]\times \{0\})=C_1$ and $h([0,1]\times \{1\})=p\big(C_1^L \times \{1\}\big)$. From (\ref{123}) and the existence of homotopy functions $g_t$ and $h_t$ above, we see that 
\begin{equation}
[C_1^L] \neq [\midpoint{\mathcal{E}_4^U}] \quad \text{in} \quad H_1(F)
\end{equation}
which contradicts equation (\ref{cont}).   \\  
We obtain that $\midpoint{\mathcal{E}_4^U}$ or $\midpoint{\mathcal{E}_6^L}$ must bound a disk in $F$ and hence we can eliminate the two triple points as in Lemma \ref{Pinch}.

\begin{figure}[H]
\centering
\captionsetup{font=scriptsize}      
\mbox{\includegraphics[scale=0.5]{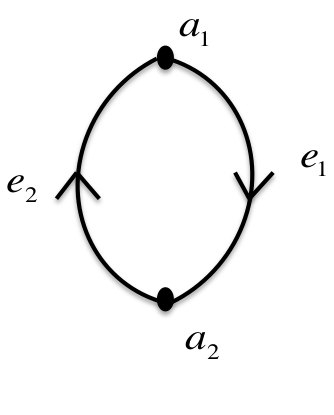}}
 \caption{}
 \label{cellu}
\end{figure} 

\end{itemize}
\end{proof}

\bibliographystyle{plain}

\end{document}